\documentclass[11pt]{amsart}
\usepackage{amsmath, amsthm, mathabx,amscd, amsfonts, amssymb, hyperref, mathrsfs, textcomp, epsfig, a4wide, graphicx, IEEEtrantools, tikz, verbatim, xypic}
\usepackage[mathscr]{euscript}

\usetikzlibrary{matrix}

\newtheorem{thm}{Theorem}
\newtheorem{prop}{Proposition}
\newtheorem{lemma}{Lemma}
\newtheorem{cor}{Corollary}
\newcommand{\spin}{\mathfrak{s}}
\theoremstyle{definition}
\newtheorem{defn}{Definition}

\theoremstyle{remark}
\newtheorem{remark}{Remark}
\newtheorem{example}{Example}

     \RequirePackage{rotating}                   
    \def\HSt{%
       \setbox0=\hbox{$\widehat{\mathit{HS}}$}
       \setbox1=\hbox{$\mathit{HS}$}
       \dimen0=1.1\ht0
       \advance\dimen0 by 1.17\ht1
       \smash{\mskip2mu\raise\dimen0\rlap{%
          \begin{turn}{180}
              {$\widehat{\phantom{\mathit{HS}}}$}
           \end{turn}} \mskip-2mu    
                \mathit{HS}
    }{\vphantom{\widehat{\mathit{HS}}}}{}}

     \RequirePackage{rotating}                   
    \def\HMt{%
       \setbox0=\hbox{$\widehat{\mathit{HM}}$}
       \setbox1=\hbox{$\mathit{HM}$}
       \dimen0=1.1\ht0
       \advance\dimen0 by 1.17\ht1
       \smash{\mskip2mu\raise\dimen0\rlap{%
          \begin{turn}{180}
              {$\widehat{\phantom{\mathit{HM}}}$}
           \end{turn}} \mskip-2mu    
                \mathit{HM}
    }{\vphantom{\widehat{\mathit{HM}}}}{}}

\newcommand{\HMf}{\widehat{\mathit{HM}}}
    \newcommand{\HSb}{\overline{\mathit{HS}}}
\newcommand{\HSf}{\widehat{\mathit{HS}}}

\newcommand{\Pin}{\mathrm{Pin}(2)}
\newcommand{\Cr}{\mathfrak{C}}
\newcommand{\Rin}{\mathcal{R}}
\newcommand{\ztwo}{\mathbb{F}}

\newcommand{\V}{\mathcal{V}}
\newcommand{\D}{\mathscr{D}}
\newcommand{\Tor}{\mathrm{Tor}}
\newcommand{\Ainf}{\mathcal{A}_{\infty}}
\newcommand{\Cj}{\hat{C}^{\jmath}_{\bullet}}
\newcommand{\x}{\mathbf{x}}
\newcommand{\y}{\mathbf{y}}
\newcommand{\p}{\mathfrak{p}}
\begin{document}

\title{$\text{PIN}(2)$-monopole Floer homology, higher compositions and connected sums} 

\author{Francesco Lin}
\address{Department of Mathematics, Massachusetts Institute of Technology} 
\email{linf@math.mit.edu}

\begin{abstract}
We study the behavior of $\Pin$-monopole Floer homology under connected sums. After constructing a (partially defined) $\Ainf$-module structure on the $\Pin$-monopole Floer chain complex of a three manifold (in the spirit of Baldwin and Bloom's monopole category), we identify up to quasi-isomorphism the Floer chain complex of a connected sum with a version of the $\Ainf$-tensor product of the modules of the summands. There is an associated Eilenberg-Moore spectral sequence converging to the Floer groups of the connected sum whose $E^2$ page is the $\mathrm{Tor}$ of the Floer groups of the summands. We discuss in detail a simple example, and use this computation to show that the $\Pin$-monopole Floer homology of $S^3$ has non trivial Massey products.\end{abstract}

\maketitle

In \cite{Lin} we introduced a set of gauge theoretic invariants of three manifolds named $\Pin$-monopole Floer homology. To each closed connected oriented three manifold $Y$ we associate in a functorial way three groups fitting in a long exact sequence
\begin{equation}\label{longex}
\dots\stackrel{i_*}{\longrightarrow}\HSt_{\bullet}(Y)\stackrel{j_*}{\longrightarrow} \HSf_{\bullet}(Y)\stackrel{p_*}{\longrightarrow} \HSb_{\bullet}(Y)\stackrel{i_*}{\longrightarrow}\dots
\end{equation}
which are read respectively \textit{H-S-to}, \textit{H-S-from} and \textit{H-S-bar}. These are also (relatively) graded topological modules over the ring 
\begin{equation*}
\Rin=\ztwo[[V]][Q]/(Q^3)
\end{equation*}
where $V$ and $Q$ have degree respectively $-4$ and $-1$ and $\ztwo$ is the field with two elements. The construction is the analogue of Manolescu's recent $\Pin$-equivariant Seiberg-Witten Floer homology (\cite{man2}) in the context of Kronheimer and Mrowka's monopole Floer homology (\cite{KM}), see \cite{Lin3} for a more detailed introduction. Unlike Manolescu's approach, which is only effective for rational homology spheres, this definition works for every three manifold (cfr. the recent work \cite{KLS} for a general construction of a version of the Seiberg-Witten spectrum).
\\
\par
Manolescu introduced his new invariants in order to disprove the longstanding Triangulation conjecture. He defined numerical invariants $\alpha,\beta$ and $\gamma$ of homology spheres to show that in the homology cobordism group $\Theta_3^H$ there are no elements of order two with Rokhlin invariant one. This is equivalent to the Triangulation conjecture being false by classic work of Galewski-Stern and Matumoto. We refer the reader to \cite{Man3} for a nice survey of the problem.
\par On the other hand not much more is known about the structure of $\Theta_3^H$. It has been shown in \cite{Fur} that it is not finitely generated but it is not known whether it has any torsion or it splits a $\mathbb{Z}^{\infty}$ summand. An interesting goal is to see whether $\Pin$-monopole Floer homology can be used to say something more precise about the structure of $\Theta_3^H$. The first goal is then to understand its behavior under connected sum, and this is the problem addressed in the present work.
\\
\par
When compared to the case of monopole Floer homology, the results are much more involved. First of all, we define higher operations satisfying the $\mathcal{A}_{\infty}$-relations to enhance the classic module structure. These algebraic structures appear naturally in the Heegaard Floer counterpart of the theory (see \cite{LOT}), and more in general when studying symplectic geometry (see for example \cite{Sei}). An analogous result for usual monopole Floer homology follows from the work in preparation of Baldwin and Bloom \cite{BB}, which actually provides a more refined version at an $\Ainf$-categorical level. Our higher compositions should be thought as the higher compositions of morphisms in their monopole category.
\par
Consider the \textit{from} Floer chain complex $\hat{C}_{\bullet}(Y,\spin)$ associated to a self-conjugate spin$^c$ structure $\spin$. This comes with a natural chain involution $\jmath$, and we denote by $\Cj (Y,\spin)$ the invariant subcomplex. The homology of this subcomplex is the Floer homology group $\HSf_{\bullet}(Y,\spin)$. In our setting, the higher compositions arise when studying special families of metrics and perturbations on cobordisms. Recall that the module structure
\begin{equation*}
\HSf_{\bullet}(Y)\otimes \Rin\rightarrow  \HSf_{\bullet}(Y)
\end{equation*}
is induced by a chain map
\begin{equation*}
m_2: \Cj(Y,2)\rightarrow \Cj(Y)
\end{equation*}
which is obtained by counting monopoles on the cobordism $I\times Y$ with a ball removed (after attaching cylindrical ends). Here 
\begin{equation*}
\Cj(Y,2)\subset \Cj(Y)\otimes \Cj(S^3)
\end{equation*}
is a subcomplex such that the inclusion is a quasi-isomorphism. The motivation is that we are working in a Morse-Bott setting so we need to consider chains which are transverse to the moduli spaces. While the map induced in homology is well defined, the map $m_2$ depends upon many choices including the ball we are removing, the metric and the perturbation. On the other hand any two such choices are chain homotopic. Analogously the ring multiplication
\begin{equation*}
\Rin\otimes \Rin\rightarrow \Rin
\end{equation*}
is induced by a chain map
\begin{equation*}
\mu_2:\Cj(S^3,2)\rightarrow \Cj(S^3)
\end{equation*}
obtained by taking $Y$ to be $S^3$ in the construction above.
\\
\par
We are ready to state our first result. Intuitively speaking, our theorem says that the chain complex $\Cj(S^3)$ has the structure of an $\Ainf$-algebra and $\Cj(Y)$ is an $\Ainf$-module over it. This is slightly imprecise because the various compositions are only defined on a subcomplex, but this heuristic will never be harmful in the present work. We refer the reader unfamiliar with the basic notions of $\Ainf$-algebra to Section \ref{algebra} for a very quick introduction to the subject. For the sake of clarity we state the theorem in a slightly imprecise way, and we refer to Section \ref{higher} for the admissibility conditions on metrics and perturbations we assume.

\begin{thm}\label{ainf}
Suppose we have chosen perturbations so that the maps $m_2$ and $\mu_2$ above are defined. Let $\mu_1$ and $m_1$ denote the differentials. There are suitable choices of data $\mathscr{D}$ (including a sequence of families of metrics and perturbations) so that for $i\geq 3$ there are higher compositions
\begin{align*}
m_i&:\Cj(S^3,i)\rightarrow \Cj(S^3)\\
\mu_i&: \Cj(Y, i)\rightarrow \Cj(S^3)
\end{align*}
defined on subcomplexes
\begin{align*}
\Cj(S^3,i)& \subset \Cj(Y)^{\otimes i}\\
\Cj(Y,i)& \subset \Cj(Y)\otimes \Cj(S^3)^{\otimes (i-1)}
\end{align*}
such that the inclusions are quasi-isomorphisms. The higher compositions respect the $\Ainf$-relations for $\Ainf$-modules over and $\Ainf$-algebra (see Section \ref{algebra} for the exact formulas), and they are independent of the choice of data $\mathscr{D}$, see Theorem \ref{invariance} for the precise statement. The same result holds for the from and bar flavors, and the chain maps $i, j$ and $p$ inducing the long exact sequence relating them can also be enhanced to higher morphisms $\{i_n\}$, $\{j_n\}$ and $\{p_n\}$, satisfying the $\Ainf$-relations for morphisms.
\par
Finally, a suitably decorated cobordism $W$ between $Y_0$ and $Y_1$ induces an $\Ainf$-module homomorphism between $\Cj(Y_0)$ and $\Cj(Y_1)$.
\end{thm}

\vspace{0.8cm}

With this additional structure in hand, after making the suitable choices we can define the chain complex
\begin{equation}\label{conn}
\Cj(Y_0,\spin_0)\boxtimes\left(\Cj(Y_1,\spin_1)\right)^{\mathrm{opp}}
\end{equation}
which is (after taking care of the transversality issues as above) a version of the $\Ainf$-tensor product of the two $\Ainf$-modules. Here $\left(\Cj(Y_1,\spin_1)\right)^{\mathrm{opp}}$ denote the opposite module of $\Cj(Y_1,\spin_1)$, which is the left $\Ainf$-module with higher operations defined by reversing the given ones, i.e.
\begin{equation}\label{opposite}
m_n^{\mathrm{opp}}(r_{n-1},\cdots, r_1,\x)=m_n(\x, r_1,\cdots, r_{n-1}).
\end{equation}
The main result of the present paper is the following.

\begin{thm}\label{main}
Given three manifolds $Y_0$ and $Y_1$ equipped with self conjugate spin$^c$ structures $\spin_0$ and $\spin_1$, there is a quasi-isomorphism
\begin{equation*}
\Phi: \Cj(Y_0,\spin_0)\boxtimes\Cj(Y_1,\spin_1)\rightarrow \Cj(Y_0\# Y_1,\spin_0\#\spin_1)\langle-1\rangle
\end{equation*}
where the angle brackets indicate the absolute grading shift. This quasi-isomorphism depends on many choices but any two of them give rise to homotopic maps. Furthermore there is a natural $\Rin$-action on the homology of the right hand side and the isomorphism induces $\Phi$ is an isomorphism of $\Rin$-modules.
\end{thm}

The module structure is obtained by defining (with the caveats as above) on $\Cj(Y)$ an $\Ainf$-bimodule structure over $\Cj(S^3)$, so that the chain complex \ref{conn} is actually an $\Ainf$-module itself.
\par
By considering a natural filtration associated to this tensor product we prove the existence of the following Eilenberg-Moore spectral sequence.

\begin{cor}\label{EM}
Given three manifolds $Y_0,Y_1$ equipped with self conjugate spin$^c$ structures $\spin_0,\spin_1$ there is a forth quadrant spectral sequence whose $E^2$ page is
\begin{equation*}
E^2_{*,*}\cong \Tor^{\Rin}_{*,*}\left(\HSf_{\bullet}(Y_0,\spin_0),\HSf_{\bullet}(Y_1,\spin_1)\right)
\end{equation*}
converging to $\HSf_{\bullet}(Y_0\# Y_1,\spin_0\#\spin_1)\langle-1\rangle$. 
\end{cor}

As for the usual Eilenberg-Moore spectral sequence, the $\mathrm{Tor}$ is taken in the category of graded modules, see Section \ref{algebra} for a quick review. The intuition behind this last result comes from the Seiberg-Witten homotopy type for rational homology spheres constructed by Manolescu (\cite{man2}). Roughly speaking, he associates to each pair $(Y_i,\spin_i)$ a $\Pin$-equivariant stable pointed homotopy type $\mathrm{SWF}(Y_i,\spin_i)$ whose $\Pin$-equivariant homology is (conjecturally) the $\Pin$-monopole Floer homology of the pair. In his approach there is an identification between the homotopy type of the connected sum and the smash products of the homotopy types of the two summands
\begin{equation*}
\mathrm{SWF}(Y_0\# Y_1,\spin_0\#\spin_1)\cong \mathrm{SWF}(Y_0,\spin_0)\wedge \mathrm{SWF}(Y_1,\spin_1).
\end{equation*}
This more topological approach has been exploited in \cite{Sto2} to provide many interesting computations for Seifert fibered spaces.
The spectral sequence described in Corollary \ref{EM} is then the Eilenberg-Moore spectral sequence computing the $\Pin$-equivariant homology of the smash product of two spaces from the $\Pin$-equivariant homology of the factors. We refer the reader to \cite{McC} for an introduction to the classical Eilenberg-Moore spectral sequence in algebraic topology.
\par
The analogous result is true in the case of usual monopole Floer homology (by just replacing the invariant Floer complexes with the total ones). In this case the $\Tor$ is taken over $\ztwo[[U]]$. In particular the spectral sequence is much simpler because the ring is a PID, and in fact it collapses at the $E^2$ page. In this sense we recover the connect sum formula for Heegaard Floer homology proved in \cite{OS2} and the one for monopole Floer homology discussed in the unpublished work \cite{BMO}. Our proof is in fact an elaboration of the proof in that work, which we briefly sketch in Section \ref{HMconn}. It is functorial in nature, and relies on the surgery exact triangle (see \cite{KMOS} and \cite{Lin2}). The key idea is that there is a natural cobordism $X^{\hash}$ from $Y_0\coprod Y_1$ to $Y_0\# Y_1$ given by a single one handle attachment, and the isomorphism in Theorem \ref{main} is given by studying the Seiberg-Witten equations on such a cobordism (and the cobordisms obtained by removing balls from it). In the whole discussion we use the \textit{from} version of Floer homology because it is the one which is well behaved with respect to cobordisms with multiple incoming ends. Indeed work of \cite{Blo}, this cobordism $X^{\hash}$ induces a map
\begin{equation*}
\HSf_{\bullet}(X^{\hash}): \HSf_{\bullet}(Y_0,\spin_0)\otimes \HSf_{\bullet}(Y_1,\spin_1)\rightarrow \HSf_{\bullet}(Y_0\#Y_1,\spin_0\# \spin_1)
\end{equation*}
while the corresponding statement for the \textit{to} version is false. In fact we can say something more about the image of this map.

\begin{cor}\label{image}
The image of the map $\HSf_{\bullet}(X^{\hash})$ in $\HSf_{\bullet}(Y_0\#Y_1,\spin_0\# \spin_1)$ corresponds to the sum
\begin{equation*}
\bigoplus_{p\in \mathbb{Z}} E^{\infty}_{p,0}
\end{equation*}
in the limit of the spectral sequence in Theorem \ref{conn}. In particular the Manolescu correction terms $\alpha,\beta,\gamma$ of homology spheres $Y_0,Y_1$ satisfy the following inequalities
\begin{align*}
\alpha(Y_0)+\alpha(Y_1)\geq & \alpha(Y_0\#Y_1)\geq \mathrm{max}\left(\alpha(Y_0)+\gamma(Y_1),\beta(Y_0)+\beta(Y_1)\right)\\
\alpha(Y_0)+\beta(Y_1)\geq & \beta(Y_0\#Y_1)\geq \beta(Y_0)+\gamma(Y_1)\\
\mathrm{min}\left(\alpha(Y_0)+\gamma(Y_1),\beta(Y_0)+\beta(Y_1)\right)\geq & \gamma(Y_0\#Y_1)\geq \gamma(Y_0)+\gamma(Y_1).
\end{align*}
\end{cor}
The last statement is also proved in \cite{Sto2} in the Seiberg-Witten Floer homology setting.
\vspace{0.5cm}
Unfortunately the $E^2$ page of the spectral sequence in Theorem \ref{conn} is generally very complicated. For example $\Tor_{*,*}(\ztwo,\ztwo)$ is non trivial in infinitely many bidegrees. Here we consider $\ztwo$ with the trivial $\Rin$-module structure. Nevertheless the spectral sequence can be used in combination with other observations (as the Gysin exact sequence, see \cite{Lin2}) to compute the $\Pin$-monopole Floer homology groups in some simple cases, e.g. the connected sum of two copies of $\Sigma(2,3,11)$. These computations can be used to show the following.

\begin{thm}\label{nonadditive}[See also \cite{Man3}]
No non-trivial linear combination of Manolescu's correction terms $\alpha,\beta,\gamma$ defines a homomorphism on the homology cobordism group $\Theta^H_3$.
\end{thm}

\begin{remark}
As a curiosity, the relation
\begin{equation*}
\lambda(Y)=\chi(\mathit{HS}_{\bullet}(Y))-\alpha(Y)+\beta(Y)-\gamma(Y)
\end{equation*}
holds, where $\lambda$ is the Casson invariant and $HS_{\bullet}(Y)$ is the reduced Floer group (i.e. the kernel of $j_*$). Recall that the latter is finite dimensional and admits an absolute $\mathbb{Z}/2\mathbb{Z}$ grading. In particular, the right hand side is additive. This follows from the analogous statement
\begin{equation*}
\lambda(Y)=\chi(\mathit{HM}_{\bullet}(Y))-d(Y)/2
\end{equation*}
showed in \cite{OSd} (whose proof works verbatim in the monopole case), and the Gysin exact sequence.\end{remark}

The higher compositions have a manifestation at the homological level in terms of Massey products. In Section \ref{massey} we discuss how to interpret higher differentials in the Eilenberg-Moore spectral sequence in terms of these Massey products. This relation has been exploited in more classical problems in algebraic topology, see for example \cite{McC}. A byproduct of our computation is that there are non-vanishing higher differentials in the spectral sequence for the connected sum of two copies of $\Sigma(2,3,11)$. This also implies the following.
\begin{cor}\label{nonformal}
$\HSf_{\bullet}(S^3)$ has nontrivial Massey products.
\end{cor}

\vspace{1cm}
\textbf{Plan of the paper.} In Section \ref{algebra} we provide a quick review of the homological algebra involved in the results, namely $\Ainf$-tensor products and the graded $\mathrm{Tor}$ functors. In Section \ref{HMconn} we review the functorial framework in which the classical connected sum formula is proved. In Section \ref{transversality} we discuss a general transversality result for moduli spaces parametrized by manifolds with corners. In Section \ref{higher} we construct the $\Ainf$-module structure on the $\Pin$-monopole Floer chain complex. Building on this, in Section \ref{formula} we prove the connect sum formula in Theorem \ref{main}. In Section \ref{massey} we discuss the relation between higher differentials, the module structure and the Massey products, and in finally in Section \ref{compute} we exhibit some explicit computations.

\vspace{1cm}
\textbf{Acknowledgements.} The author would like to thank John Baldwin, Jonathan Bloom, Lucas Culler, his advisor Tom Mrowka, Ciprian Manolescu, Matthew Stoffregen and Umut Varolgunes for the nice and useful discussions. This work was partially supported by the NSF grant DMS-0805841.

\vspace{1cm}

\section{A quick review of homological algebra}\label{algebra}
In this section we briefly review the homological algebra that could be unfamiliar to people working in Floer homology for three-manifolds. We discuss the definition of the graded $\Tor$ functor that appears in the statement of Corollary \ref{EM}, and the basic notions regarding $\Ainf$-algebra and modules. Regarding the latter, as pointed out in our setting the setup will be slightly more complicated because the constructions will be performed in a geometric fashion, and we will have to deal with transversality issues. On the other hand the homological algebra discussed in the second part of this section motivates our geometric constructions, and the actual formulas will be indeed the same. We will always suppose that our ground ring is $\ztwo$, the field with two elements. This will allow us to forget about sign conventions and flatness issues.
\\
\par
\subsection{The graded $\Tor$ functor}
The graded $\Tor$ is the object that arises as the $E^2$-page of the Eilenberg-Moore spectral sequence. Our treatment will follow Chapter $7$ of \cite{McC}, with less details. Consider a graded algebra $R$ over $\ztwo$, and let $M$ and $N$ be respectively a right and left graded module over $R$. We suppose that $M$ and $N$ are graded by a coset of $\mathbb{Z}$ in $\mathbb{Q}$. From this data we can form the bigraded $R$-module $\Tor^{R}_{*,*}(M,N)$ (where the action of $R$ affects the second grading accordingly) defined as follows. Take a projective resolution of $N$
\begin{equation*}
0\stackrel{\epsilon}{\longleftarrow}N \stackrel{\delta}{\longleftarrow} P^0 \stackrel{\delta}{\longleftarrow} P^1 \stackrel{\delta}{\longleftarrow}\cdots \stackrel{\delta}{\longleftarrow} P^{n-1} \stackrel{\delta}{\longleftarrow} P^n \stackrel{\delta}{\longleftarrow}\cdots
\end{equation*}
where the $P^{n}$ are projective graded modules over $R$ and the maps $\delta$ are grading preserving. In particular for each $j$ we get an exact sequence of $\ztwo$-vector spaces by considering the homogenous components of degree $j$:
\begin{equation*}
0\stackrel{\epsilon}{\longleftarrow}N^j \stackrel{\delta}{\longleftarrow} (P^0)^j \stackrel{\delta}{\longleftarrow} (P^1)^j \stackrel{\delta}{\longleftarrow}\cdots \stackrel{\delta}{\longleftarrow} (P^{n-1})^j \stackrel{\delta}{\longleftarrow} (P^n)^j \stackrel{\delta}{\longleftarrow}\cdots
\end{equation*}
Consider the complex $(P^{\bullet}, \delta)$ given by
\begin{equation*}
P^0 \stackrel{\delta}{\longleftarrow} P^1 \stackrel{\delta}{\longleftarrow}\cdots \stackrel{\delta}{\longleftarrow} P^{n-1} \stackrel{\delta}{\longleftarrow} P^n \stackrel{\delta}{\longleftarrow}\cdots\end{equation*}
which has homology $H(P^{\bullet}, \delta)$ isomorphic to $N$. We can form the graded differential module $(\mathrm{total}(P^{\bullet}), \delta)$ setting
\begin{equation*}
\mathrm{total}(P^{\bullet})^j=\bigoplus_{m+n=j} (P^m)^n
\end{equation*}
which is a differential graded module over $R$. Here the action of $R^i$ sends $(P^m)^n$ to $(P^m)^{n+i}$. For the right $R$-module we can then define the $\Tor$ group as
\begin{equation*}
\Tor^{R}(M,N)=H(M\otimes_{R}\mathrm{total}(P^{\bullet}), \delta)
\end{equation*}
This $R$-module is well defined, meaning that up to isomorphism it is independent of the choice of the proper projective resolution of $N$. It can also be computed by tensoring a graded projective resolutions of $M$ with $N$, or by tensoring two graded projective resolutions. 
\par
The key point is that $\Tor_{R}(M,N)$ is a bigraded object, where the elements coming from $M^m\otimes_{\Gamma}(P^{-i})^n$ have bidegree $(-i, m+n)$. The first degree is the homological degree and the second one is the internal degree.
\vspace{0.5cm}
\begin{example}\label{S2311}
Consider the $\Rin$-module $M$ given by the three towers
\begin{equation*}
\ztwo[[V]]\langle-2\rangle\oplus\ztwo[[V]]\langle-3\rangle\oplus\ztwo[[V]]
\end{equation*}
where the action of $Q$ (which has degree $-1$) is an injection from the first tower to the second and from the second tower to the third. We then define the $\Rin$-modules
\begin{align*}
P_{2n}&=\Rin\langle-3n\rangle \oplus\Rin\langle-3n-2\rangle\\
P_{2n+1}&=\Rin\langle-3n-1\rangle \oplus\Rin\langle-3n-4\rangle.
\end{align*}
Define for each $n\geq 0$ the maps
\begin{align*}
\delta_{2n}&:P_{2n}\rightarrow P_{2n-1}\\
(1,0)&\mapsto (Q^2,0)\\
(0,1)&\mapsto (V,Q)
\end{align*}
and
\begin{align*}
\delta_{2n+1}&:P_{2n+1}\rightarrow P_{2n}\\
(1,0)&\mapsto (Q,0)\\
(0,1)&\mapsto (V,Q^2).
\end{align*}
If we consider $\varepsilon: P_0\rightarrow M$ obtained by sending $(1,0)$ to $(0,0,1)$ and $(1,0)$ to $(1,0,0)$ then the sequence
\begin{equation*}
0\stackrel{\epsilon}{\longleftarrow}M \stackrel{\delta}{\longleftarrow}P^0 \stackrel{\delta}{\longleftarrow} P^1 \stackrel{\delta}{\longleftarrow}\cdots \stackrel{\delta}{\longleftarrow} P^{n-1} \stackrel{\delta}{\longleftarrow} P^n \stackrel{\delta}{\longleftarrow}\cdots
\end{equation*}
is a graded projective resolution of the $\Rin$-module $M$. This projective resolution has the nice property of being two periodic: shifting the first degree by two is the same as lowering the second degree by $-3$. Define the $\Rin$-module $N$ given by 
\begin{equation*}
\ztwo[[V]]\langle-4\rangle\oplus\ztwo[[V]]\langle-1\rangle\oplus\ztwo[[V]]\langle-2\rangle
\end{equation*}
where again the action of $Q$ is an injection from the first tower to the second and from the second to the third. We then have in particular that
\begin{equation*}
\mathrm{ker}d_{2n}\cong N\langle-3n\rangle\quad\text{and}\quad\mathrm{ker}d_{2n+1}\cong M\langle-3-3n\rangle.
\end{equation*}
From this projective resolution of $M$ it is easy to also reconstruct a $2$-periodic graded projective resolution of $N$.
\end{example}

\vspace{0.5cm}

\begin{example}\label{FF}
Consider the trivial $\Rin$-module $\ztwo$ in degree zero. It is straightorward to construct a projective resolution of this module from the discussion above. Indeed we can pick $P_0$ to be $\Rin$ with the map $\varepsilon:\Rin\rightarrow\ztwo$ sending $1$ to the generator of $\ztwo$. The kernel of this map is the module $N\langle-1\rangle$ so we can exploit the construction above. From this we can compute that
\begin{equation*}
\Tor^{\Rin}_{i,j}(\ztwo,\ztwo)\cong \ztwo
\end{equation*}
in the cases 
\begin{itemize}
\item $(i,j)=(0,0)$;
\item $i=2n$ for $n\geq1$ and $j=-3n$ or $-3n-2$;
\item $i=2n+1$ and $j=-1-3n$ or $-4-3n$,
\end{itemize}
and zero otherwise. When thinking about the total degree, $\Tor^{\Rin}_{*,*}(\ztwo,\ztwo)$ has rank two over $\ztwo$ on each degree. The following figure shows the placement of the non trivial summands for $j\leq 6$.

\begin{center}
\begin{tikzpicture}
\matrix (m) [matrix of math nodes,row sep=0.1em,column sep=0.5em,minimum width=0.1em]
  { \ztwo& \cdot & \cdot &\cdot& \cdot& \cdot\\
  \cdot & \ztwo &\cdot & \cdot & \cdot & \cdot\\
  \cdot & \cdot & \cdot & \cdot & \cdot & \cdot\\
  \cdot &\cdot& \ztwo &\cdot &\cdot&\cdot\\
  \cdot & \ztwo & \cdot & \ztwo & \cdot & \cdot\\
  \cdot &\cdot& \ztwo &\cdot &\cdot&\cdot\\
  \cdot &\cdot&\cdot&\cdot&\ztwo&\cdot\\
  \cdot& \cdot&\cdot&\ztwo &\cdot& \ztwo \\
    \cdot &\cdot&\cdot&\cdot&\ztwo&\cdot\\
    \cdot & \cdot & \cdot & \cdot & \cdot & \cdot\\
    \cdot & \cdot & \cdot & \cdot & \cdot & \ztwo\\};
\end{tikzpicture}
\end{center}

Unfortunately the result is infinite dimensional, which makes computations with the Eilenberg-Moore spectral sequence very complicated. Nevertheless it is important to notice that the placement of the non trivial groups allows the spectral sequence to converge to a finite dimensional group: the differentials $d_2,d_3,d_4$ can be non trivial.
\end{example}

\vspace{0.5cm}
There is a nice general way to construct a graded projective resolution of an $R$-module called the \textit{bar construction}. Here we suppose that the ring $R$ is unital and that $R^0$, the homogeneous components in degree $0$, are identified with the ground ring $\ztwo$ (as in the case of $\Rin$), and we define
\begin{equation*}
\bar{R}=\bigoplus_{j>0} R^j.
\end{equation*}
Define
\begin{equation*}
B^{n}(R,N)=R\otimes\underbrace{\bar{R}\otimes\cdots \otimes\bar{R}}_\text{$n$ times}\otimes N,
\end{equation*}
where the tensor products are taken over $\ztwo$. We can write an element of this module as $[\gamma_0\lvert \gamma_1\lvert \cdots \lvert \gamma_n]a$. This is a graded projective module over $R$. These $R$-modules can be assembled in a graded projective resolution by introducing the homological differential
\begin{equation*}
\delta: B^{n}(R,N)\rightarrow B^{n-1}(R,N)
\end{equation*}
given by
\begin{equation*}
\delta([\gamma_0\lvert\gamma_1\lvert  \cdots \lvert \gamma_n]x)=\sum_{i=1}^n [\gamma_0\lvert \gamma_1\lvert \cdots\lvert \gamma_{i-1}\cdot\gamma_i\lvert \gamma_{i+1}\lvert \cdots \lvert \gamma_n] \x+[\gamma_0\lvert \gamma_1\lvert \cdots \lvert \gamma_{n-1}](\gamma_n\cdot \x).
\end{equation*}
One can then check that the complex
\begin{equation*}
0\stackrel{\epsilon}{\longleftarrow}N \stackrel{\delta}{\longleftarrow} B^0(R,N) \stackrel{\delta}{\longleftarrow} B^1(R,N) \stackrel{\delta}{\longleftarrow}\cdots \stackrel{\delta}{\longleftarrow} B^n(R,N) \stackrel{\delta}{\longleftarrow}\cdots,
\end{equation*}
where the map $\varepsilon: \Gamma\otimes_{\ztwo} N\rightarrow N$ is given by the module structure, is a graded projective resolution of the $R$-module $N$. We call this the \textit{bar resolution} of $N$. We can define
\begin{equation*}
B^{n}(M,R, N)= M\otimes_{R} B^{n}(R,N)= M\otimes_{\ztwo}\underbrace{\bar{R}\otimes_{\ztwo}\cdots \otimes_{\ztwo}\bar{R}}_\text{$n$ times}\otimes_{\ztwo} N
\end{equation*}
with differential $1\otimes_{R} \delta$. We have then by definition we have that
\begin{equation*}
H(\mathrm{total}(B^{\bullet}(M,R,N)),\delta)=\Tor^{R}(M,N).
\end{equation*}
This description is particularly interesting in our case because a variant of the bar resolution will arise as the $E^1$-page of the spectral sequence described in Corollary \ref{EM}, see Lemma \ref{algEM}. The observation is that under our assumption the bar construction can be performed by replacing $\bar{R}$ by $R$.

\begin{example}
The groups described in Example \ref{FF} can be described when thought as the homology of the bar resolution. Denote by $\x$ the generator of $\ztwo$. For example the group at $(0,0)$ is generated by the class of $\x[ ]\x$. The groups at $(1,-1)$ and $(1,-4)$ are generated respectively by $\x[Q]\x$ and $\x[V]\x$, while those at $(2,-3)$ and $(2,-5)$ by $\x[Q\lvert Q^2]\x$ and $\x[V\lvert Q]\x+ \x[Q\lvert V]\x$. 
\end{example}

\vspace{0.8cm}

\subsection{Basics on $\Ainf$-algebras and modules} We now provide a quick (and incomplete) discussion of the basics of $\Ainf$-structures, mostly in order to fix notation and adapt the discussion to our case. For a more thorough introduction (including motivations) we refer to \cite{Kel} and \cite{Val}. As above, we will always suppose that our ground ring is $\ztwo$.
\\
\par
An $\Ainf$-algebra $\mathcal{A}$ is a graded vector space $A$ together with maps
\begin{equation*}
\mu_i: A^{\otimes i}\rightarrow A[2-i]
\end{equation*}
defined for each $i\geq 1$ so that the compatibility conditions
\begin{equation}\label{infalg}
\sum_{i+j=n+1}\sum_{l=1}^{n-j+1}\mu_i(a_1,\dots ,a_{l-1}, \mu_j(a_l,\dots ,a_{l+j-1}), a_{l+j},\dots ,a_n)=0
\end{equation}
hold for each $n\geq 1$. Here to simplify the notation we denote $\mu_n(a_1\otimes\cdots\otimes a_n)$ by $\mu_n(a_1,\dots,a_n)$.
In particular $\mu_1$ is a differential, $\mu_2$ is a chain map and $\mu_3$ is a chain homotopy between $\mu_2(\mu_2\otimes 1)$ and $\mu_2(1\otimes \mu_2)$. This implies that the homology $H_*(\mathcal{A})$ is a graded algebra over $\ztwo$, where the multiplication is the one induced by $\mu_2$.
\\
\par
A right $\Ainf$-module $\mathcal{M}$ over $\mathcal{A}$ is a vector space $M$ graded by a coset of $\mathbb{Z}$ in $\mathbb{Q}$ together with maps
\begin{equation*}
m_{i}: M \otimes A^{\otimes i-1}\rightarrow M[2-i]
\end{equation*}
satisfying the compatibility relations
\begin{align}\label{infmod}
0&=\sum_{i+j=n+2} m_i(m_j(\x , a_1,\dots , a_{j-1}), a_j ,\dots , a_n)\\
& +\sum \sum m_i(\x, a_1,\dots , \mu_j(a_l,\dots , a_{l+j-1}),\dots ,a_n)
\end{align}
As above, $m_1$ is a differential, $m_2$ is a chain map and $m_3$ is a chain homotopy between $m_2(m_2\otimes 1)$ and $m_2(1\otimes \mu_2)$. The homology $H_*(\mathcal{M})$ is a graded module over $H(\mathcal{A})$, the multiplication being induced by $m_2$.
\par
The definition of a left $\Ainf$-module is analogous. Given a right $\Ainf$-module $\mathcal{M}$, the opposite module $\mathcal{M}^{\mathrm{opp}}$ is the left $\Ainf$-module obtained by reversing all the operations as in equation (\ref{opposite}).
\\
\par
Even though the higher operations $\mu_i$ are not chain maps for $i\geq 3$, so that they do not define maps in homology, they have partial incarnations on $H_*(\mathcal{A})$ in terms of Massey products. In particular suppose we are given cycles $a_1, a_2$ and $a_3$ of degrees $d_i$ so that
\begin{equation*}
[a_1]\cdot[a_2]=0\quad[a_2]\cdot[a_3]=0.
\end{equation*}
Let $s_1$ and $s_2$ be chains so that
\begin{equation*}
\mu_1(s_1)=\mu_2(a_1,a_2)\quad \mu_1(s_2)=\mu_2(a_2,a_3).
\end{equation*}
We define then the Massey product of $[a_1],[a_2]$ and $[a_3]$ as the homology class
\begin{equation*}
\langle[a_1],[a_2],[a_3]\rangle=[\mu_3(a_1,a_2,a_3)+\mu_2(s_1,a_3)+\mu_2(a_1,s_2)]
\end{equation*}
which is a well defined class in 
\begin{equation*}
H_{d_1+d_2+d_3-1}(\mathcal{A})/([a_1]\cdot H_{d_2+d_3-1}(\mathcal{A})+H_{d_1+d_1-1}(\mathcal{A})\cdot [a_3]).
\end{equation*}
The analogous formulas determine partially defined triple products in the cohomology of a right $\Ainf$-module $\mathcal{M}$.
\par
There are of course incarnations of all the compositions $m_n$ when all the lower degree Massey product vanish in a consistent way. In the present work we will need the case for $n=4$, which we discuss in detail now. Suppose we are cycles $a_1, a_2, a_3$ and $a_4$ so that
\begin{align*}
\mu_2(a_1,a_2)&=\mu_1(s_1)\\ \mu_2(a_2,a_3)&=\mu_1(s_2)\\ \mu_2(a_3,a_4)&=\mu_1(s_3).
\end{align*}
for some chains $s_i$. Suppose furthermore that
\begin{align*}
\mu_3(a_1,a_2,a_3)+\mu_2(s_1,a_3)+\mu_2(a_1,s_2)&=\mu_1(t_1)\\
\mu_3(a_2,a_3,a_4)+\mu_2(s_2,a_4)+\mu_2(a_1,s_3)&=\mu_1(t_2)
\end{align*}
for some $t_1$ and $t_2$. We can then define the homology class of
\begin{align*}
\mu_4(a_1,a_2,a_3,a_4)&+\mu_3(s_1,a_3,a_4)+\mu_3(a_1,s_2,a_4)+\mu_3(a_1,a_2,s_3)\\
&+\mu_2(t_1,a_4)+\mu_2(a_1,t_2)+\mu_2(s_1,s_3).
\end{align*}
The fourfold Massey product, denoted by $\langle[a_1],[a_2],[a_3],[a_4]\rangle$, is the set of homology classes in $H_{d_1+d_2+d_3+d_4-2}(\mathcal{A})$ arising from the construction above.
\\
\par
Given two $\Ainf$-algebras $\mathcal{A}$ and $\mathcal{B}$ with respective operations $\{\mu_k\}$ and $\{\nu_k\}$, an $\Ainf$-morphism $f$ between them consists of a collection of maps
\begin{equation*}
f_n: A^{\otimes n}\rightarrow B[n-1]
\end{equation*}
for each $n\geq1$ satysfying the relations
\begin{align}\label{amorp}
0&=\sum_{k\geq 1}\sum_{i_1+\dots i_k=n} \nu_k\left(f_{i_1}(a_1,\dots, a_{i_1}), f_{i_2}(a_{i_1+1},\dots ,a_{i_1+i_2}) ,\dots , f_{i_k}(a_{i_1+\dots +i_{k-1} +1},\dots, a_{n})\right)\\
&+\sum_{k+l=n+1}\sum_{1\leq j\leq k} f_k(a_1,\dots, \mu_l(a_j,\dots,a_{k+l-1}),\dots ,a_n).
\end{align}
We say that an $\Ainf$-morphism $f$ is nullhomotopic if there are maps
\begin{equation*}
h_j: A^{\otimes j}\rightarrow B[-j]
\end{equation*}
such that 
\begin{align*}
f_n(a_1,\dots, a_n)&=+\sum_{k\geq 1}\sum_{i_1+\dots i_k=n} \nu_k\left(h_{i_1}(a_1,\dots, a_{i_1}),\dots , h_{i_k}(a_{i_1+\dots +i_{k-1} +1},\dots, a_{n})\right)\\
&+\sum_{k+l=n+1}\sum_{1\leq j\leq k} h_k(a_1,\dots, \mu_l(a_j,\dots,a_{k+l-1}),\dots ,a_n),
\end{align*}
and that two maps $f_0$ and $f_1$ are homotopic if their difference is nullhomotopic. The composition of two $\Ainf$-morphisms $g$ and $f$ is given by
\begin{equation*}
(g\circ f)_n=\sum_{k\geq_1} \sum_{i_1+\dots+i_k=n} g_k\left(f_{i_1}(a_1,\dots, a_{i_k}),\dots , f_{i_k}(a_{i_1+\dots+i_{k-1}+1},\dots , a_n)\right)
\end{equation*}
We say that two $\Ainf$-algebras $\mathcal{A}$ and $\mathcal{B}$ are weakly homotopy equivalent if there are maps $f:\mathcal{A}\rightarrow \mathcal{B}$ and $g:\mathcal{B}\rightarrow \mathcal{A}$ so that the compositions $g\circ f$ and $f\circ g$ induce the identity in homology.
\\
\par
Given two right $\Ainf$-modules $\mathcal{M}$ and $\mathcal{M}'$ over $\mathcal{A}$, an $\Ainf$-homomorphism $f$ consists of a collection of maps
\begin{equation*}
f_{i}: M \otimes A^{\otimes i-1}\rightarrow M'[1-i]
\end{equation*}
satisfying the compatibility relations
\begin{align*}
0&=\sum_{i+j=n+1} m'_i(f_j(\x,a_1,\dots, a_{j-1}),\dots , a_{n-1})\\
& +\sum_{i+j=n+1} f_i(m_j(\x,a_1,\dots, a_{j-1}),\dots , a_{n-1})\\
&=\sum_{j=1}^{n-1} \sum_{l=1}^{n-1-j} f_i(\x,a_1,\dots, \mu_j(a_l,\dots , a_{l+j-1}),\dots ,a_{n-1})
\end{align*}
For any module $\mathcal{M}$ the identity morphism $\mathbb{I}$ is defined to be
\begin{align*}
\mathbb{I}_1(\x)&=\x\\
\mathbb{I}_i(\x|a_1|\cdots |a_{i-1}) &=0\text{ for }i>0.
\end{align*}
Given $\Ainf$-morphisms $f$ from $\mathcal{M}$ to $\mathcal{M}'$ and $g$ from $\mathcal{M}'$ to $\mathcal{M}''$ we define their composition $g\circ f$ to be
\begin{equation*}
(g\circ f)_n(\x, a_1,\dots , a_{n-1})=\sum_{i+j=n+1} g_j(f_i(\x,a_1,\dots, a_{i-1}),\dots, a_{n-1})
\end{equation*}
We say that a morphism $f$ is nullhomotopic if there are maps
\begin{equation*}
h_i: M\otimes A^{\otimes i-1}\rightarrow M'[-i]
\end{equation*}
with the property that
\begin{align}\label{ahom}
f_n(\x|a_1,\dots ,a_{n-1})&=\\
&\sum_{i+j=n+1}h_i(m_j(\x,a_1,\cdots , a_{j-1}),\cdots , a_{n-1})\\
&+\sum_{i+j=n+1} m_i'(h_j(\x,a_1,\dots , a_{j-1}),\cdots ,a_{n-1})\\
&+\sum_{i+j=n+1}\sum_{l=1}^{n-j} h_i(\x,a_1,\dots ,\mu_j(a_l,\dots , a_{l+j-1},\dots,a_{n-1}).
\end{align}
Two morphisms $f,g:\mathcal{M}\rightarrow \mathcal{M}'$ are homotopic if their difference is nullhomotopic. Two $\Ainf$-modules $\mathcal{M}$ and $\mathcal{M}'$ are homotopy equivalent if there are morphisms
\begin{equation*}
f:\mathcal{M}\rightarrow \mathcal{M}',\qquad g:\mathcal{M}'\rightarrow \mathcal{M}
\end{equation*}
such that $f\circ g$ and $g\circ f$ are homotopy equivalent to the identity.

\vspace{0.5cm}

Given $\mathcal{M}$ and $\mathcal{N}$, respectively a right and a left $\Ainf$-module over $\mathcal{A}$, we define their $\Ainf$-tensor product as the chain complex $\mathcal{M}\boxtimes\mathcal{N}$ whose underlying graded vector space is given by
\begin{equation*}
\bigoplus_{i=0}^{\infty} M\times (A[1])^{\otimes i} \otimes N
\end{equation*}
and the differential is given by
\begin{align*}
\partial[\x\lvert a_1\lvert\dots \lvert a_n\lvert \y]&:= \sum_{i+1}^{n+1} [m_i(\x, a_1,\dots ,a_{i_1})\lvert\cdots\lvert a_n\lvert\y]\\
& +\sum_{i=1}^n\sum_{l=1}^{n-i+1} [\x\lvert a_1\vert\cdots \lvert\mu_i(a_l,\dots ,a_{l+i-1})\lvert\cdots\lvert a_n\lvert \y]\\
& +\sum [\x| a_1|\cdots | m_i(a_{n-i+2},\dots , a_n,y)].
\end{align*}
Here we use the same notation we adopted for the bar resolution. The main homological algebra result we need is the following.
\begin{lemma}\label{algEM}
Suppose $H_*(\mathcal{A})$ is a unital algebra whose degree zero part is identified with $\ztwo$. Then there is a spectral sequence whose $E^2$-page is isomorphic to
\begin{equation*}
\Tor_{*,*}^{H_*(\mathcal{A})}(H_*(\mathcal{M}),H_*(\mathcal{N}))
\end{equation*}
which converges to $H_*(\mathcal{M}\boxtimes\mathcal{N}$)
\end{lemma}

\begin{proof}
The $\Ainf$ tensor product is naturally filtered by the groups
\begin{equation*}
F_n=\bigoplus_{i=0}^{n} M\times (A[1])^{\otimes i} \otimes N
\end{equation*}
and the $E^1$ page is exactly the tensor product over $\ztwo$ of the (unreduced) bar resolution of $H_*(\mathcal{M})$ with $H_*(\mathcal{N})$.
\end{proof}

From this point of view it is clear that the higher differentials can be described in terms of the higher multiplications of the $\Ainf$-modules involved. Indeed it is classically known that the differentials in the Eilenberg-Moore spectral sequence can be described in term of Massey products, see \cite{McC}.
\\
\par
Finally, suppose that in we are in the case that $\mathcal{M}$ is actually and $\Ainf$-bimodule, i.e. there are maps
\begin{equation*}
m_{i,j}: A^{\otimes i-1}\otimes M\otimes A^{\otimes j-1}\rightarrow M[1-i-j]
\end{equation*}
so that
\begin{align*}
0 &=\sum_{l=1}^{i}\sum_{k=1}^j m_{i-l+1,j-k+1}(a_{i-1}',\dots, m_{l,k}(a_{l-1}',\dots,a_1' \x,a_1,\dots, a_{k-1}) ,\dots a_{j-1})
\\
&+\sum_{k=1}^{j-1}\sum_{l=1}^{j-1-k} m_{i,j-l+1}(a_{i-1}',\dots a_1',\x, \dots,\mu_l(a_k,\dots, a_{k+l-1}),\dots a_{j-1}) \\
&+\sum_{k=1}^{i-1}\sum_{l=1}^{i-1-k} m_{i-l+1,j}(a_{i-1}',\dots ,\mu_l(a_{k+l-1}',\dots, a_k'),\dots \x,\dots, a_{j-1}).
\end{align*}
In particular the usual right $\Ainf$-module operations $m_j$ are given by $m_{1,j}$. Then the $\Ainf$-tensor product $\mathcal{M}\boxtimes\mathcal{N}$ has a natural left $\Ainf$-module structure given by
\begin{equation*}
m_n(a_n,\dots, a_1, [\x\lvert b_1\lvert\cdots\lvert b_k\lvert \y])=\sum_{l=1}^{k+1} [m_{n,l}(a_n,\dots, \x, \dots, b_{l-1})\lvert b_l\lvert \cdots\lvert \y].
\end{equation*}
It is an easy exercise to check that these actually satisfy the $\Ainf$-relations.

\vspace{1cm}

\section{Floer functors and the connected sum formula}\label{HMconn}

In this section we review the content of the unpublished work \cite{BMO}. We discuss a nice functorial picture in which our invariants fit, and discuss how it can be used to prove the connected sum formula for usual monopole Floer homology. Our proof will follow the same approach but will be more complicated from a technical viewpoint, so first discuss the main ideas in a simpler case before dwelling into the details.
\\
\par
We first introduce the categorical framework in which we work.
\begin{defn}
We define the \textit{surgery precategory} \textsc{sur} as follows. It objects are closed connected oriented three manifolds $Y$. Morphisms between $Y_0$ to $Y_1$ are given by cobordisms $W$ which are either $[0,1]\times Y_0$ or obtained from it by a single two handle attachment, up to diffeomorphism.
\end{defn}

The surgery precategory is the natural object on which one studies surgery exact triangles, which we now quickly review. Suppose we are given a three manifold $Z$ with torus boundary, and consider simple closed curves $\mu_i$ for $i=1,2,3$ on $\partial Z$ so that
\begin{equation*}
\mu_i\cdot \mu_{i+1}=-1.
\end{equation*}
The indices here are interpreted cyclically. Let $Y_i$ the manifold obtained from $Z$ by Dehn filling along $\mu_i$. There is a natural cobordism $W_i$ from $Y_i$ to $Y_{i+1}$ given by a single $2$-handle attachment.

\begin{defn}
We say that a triple of three manifolds $Y_0,Y_1,Y_2$ is a \textit{surgery triple} in the surgery precategory \textsc{sur} if it arises from the construction above (up to diffeomorphism).
\end{defn}

It is then proved in \cite{KMOS} that the triangle involving the monopole Floer homology groups
\begin{center}
\begin{tikzpicture}
\matrix (m) [matrix of math nodes,row sep=2em,column sep=1.5em,minimum width=2em]
  {
  \HMt_{\bullet}(Y_2) && \HMt_{\bullet}(Y_3)\\
  &\HMt_{\bullet}(Y_1) &\\};
  \path[-stealth]
  (m-1-1) edge node [above]{$\HMt_{\bullet}(W_2)$} (m-1-3)
  (m-2-2) edge node [left]{$\HMt_{\bullet}(W_1)$} (m-1-1)
  (m-1-3) edge node [right]{$\HMt_{\bullet}(W_3)$} (m-2-2)  
  ;
\end{tikzpicture}
\end{center}
is exact. This motivates the following definition.
\begin{defn}
A \textit{Floer functor} is a functor
\begin{equation*}
\mathscr{F}: \textsc{sur} \rightarrow \mathcal{C},
\end{equation*}
where $\mathcal{C}$ is an abelian category, which sends surgery triples to exact triangles.
\end{defn}

\begin{remark}
The surgery exact triangle in diffeomorphism invariant in the following sense. Suppose we are given a three manifold $Y_1'$ together with a diffeomorphism $\phi$ to $Y_1$. Then we get a canonical isomorphism
\begin{equation*}
\Psi: \HMt_{\bullet}(Y_1')\rightarrow \HMt_{\bullet}(Y_1)
\end{equation*}
obtained from the product $[0,1]\times Y_1'$ by parametrizing the end $\{1\}\times Y_1'$ via $\phi^{-1}$. Similarly we can use this new parametrization to obtain a cobordism $W_1'$ from $Y_1'$ to $Y_2$, and $W_3'$ from $Y_3$ to $Y_1'$. The induced map then fit in the commutative diagram
\begin{center}
\begin{tikzpicture}
\matrix (m) [matrix of math nodes,row sep=2em,column sep=1.5em,minimum width=2em]
  {
  \HMt_{\bullet}(Y_2) && \HMt_{\bullet}(Y_3)\\
  &\HMt_{\bullet}(Y_1) &\\
  &\HMt_{\bullet}(Y_1') &\\};
  \path[-stealth]
  (m-1-1) edge node [above]{$\HMt_{\bullet}(W_2)$} (m-1-3)
  (m-2-2) edge node [left]{$\HMt_{\bullet}(W_1)$} (m-1-1)
  (m-1-3) edge node [right]{$\HMt_{\bullet}(W_3)$} (m-2-2)  
  	      edge node [right]{$\HMt_{\bullet}(W_3')$} (m-3-2)  
  (m-3-2) edge node [right]{$\Phi$} (m-2-2)
  	edge node [left]{$\HMt_{\bullet}(W_1')$} (m-1-1)
  ;
\end{tikzpicture}
\end{center}
Hence even though for many constructions in the upcoming section we will work with three manifolds \textit{not} up to diffeomorphism, for the upcoming result this will be of no harm. 
\end{remark}

The key result regarding Floer functors is the following. It finds its roots in the work of Floer for instanton Floer homology (\cite{BD}).
\begin{thm}[\cite{BB},\cite{BMO}]\label{floerf}
Suppose we are given two Floer functors
\begin{equation*}
\mathscr{F}, \mathscr{G}: \textsc{sur} \rightarrow \mathcal{C},
\end{equation*}
together with a natural transformation
\begin{equation*}
\Psi: \mathscr{F}\rightarrow \mathscr{G}.
\end{equation*}
If $\Psi(S^3)$ induces an isomorphism from $\mathscr{F}(S^3)$ to $\mathscr{G}(S^3)$, then $\Psi(Y)$ is an isomorphism from $\mathscr{F}(Y)$ to $\mathscr{G}(Y)$ for every three manifold $Y$.
\end{thm}
\begin{proof}
Following \cite{BB} (see also \cite{Cul}), consider the set $\mathcal{X}$ of diffeomorphism classes of three manifolds. We say that a subset $A\subset\mathcal{X}$ is a subspace if whenever we have a surgery triangle $Y_0,Y_1, Y_2$ so that two of these manifolds lie in $A$ then also the third lies in $A$. Then we have that the smallest subspace containing $S^3$ is the whole $\mathcal{X}$. This fact, together with the five lemma applied iteratively to the exact triangles in $\mathcal{C}$, implies the result.
\end{proof}

\vspace{0.5cm}
Of course the applicability of this result in practical contexts is limited by the fact that we need to have an actual natural trasformation between the two theories. In particular, we cannot use it to provide an axiomatic characterization of monopole Floer homology or its equivalence with other theories satisfying the same axioms (as Heegaard Floer homology). On the other hand it turns out to be very useful to study the connected sum formula in monopole Floer homology, as we now discuss. We first recall the result.
\begin{thm}[\cite{BMO}]\label{connectusual}
Consider two pairs $(Y_0,\spin_0)$ and $(Y_1,\spin_1)$. There is an isomorphism of graded $\ztwo[[U]]$-modules between $\HMf_{\bullet}(Y_0\# Y_1, \spin_0\#\spin_1)$ and the mapping cone of the degree $-1$ map
\begin{equation*}
\HMf_{\bullet}(Y_0,\spin_0)\otimes\HMf_{\bullet}(Y_1,\spin_1)\stackrel{1\otimes U+U\otimes 1}{\longrightarrow}\left(\HMf_{\bullet}(Y_0,\spin_0)\otimes\HMf_{\bullet}(Y_1,\spin_1)  \right)\langle 1\rangle.
\end{equation*}
\end{thm}

It is not hard to check that the mapping cone of the statement is isomorphic (in a functorial way) as a $\ztwo[[U]]$-module to
\begin{equation*}
\Tor^{\ztwo[[U]]}_*\left(\HMf_{\bullet}(Y_0,\spin_0),\HMf_{\bullet}(Y_1,\spin_1)  \right)\langle -1\rangle
\end{equation*} 
so that this theorem is equivalent to the analogous one in Heegaard Floer homology (\cite{OS2}).
\\
\par
To prove this theorem with the functorial approach discussed above we need to introduce extra decorations in order to make the constructions natural.
\begin{defn}
We define the \textit{based surgery precategory} $\textsc{sur}_*$ as follows. Its objects are triples $(Y, \iota^U,\iota^{\hash})$ where $Y$ is a closed connected oriented three manifolds $Y$ together with smooth embeddings with disjoint images
\begin{equation*}
\iota^U,\iota^{\hash}: B^3\rightarrow Y,
\end{equation*}
where $B^3$ is the closed unit ball.
Given two such objects $(Y_0,\iota^U_0,\iota^{\hash}_0)$ and $(Y_1,\iota^U_1,\iota^{\hash}_1)$, consider the pairs $(W,\iota^U,\iota^{\hash})$ where:
\begin{itemize}
\item $W$ is a cobordism from $Y_0$ to $Y_1$ obtained from $[0,1]\times Y_0$ by a single two handle attachment, so that the handle is attached away from the balls
\begin{equation*}
\{1\}\times \iota^U_0(B)\amalg\{1\}\times \iota^{\hash}_0(B);
\end{equation*}
\item $\iota^U$ is a proper embedding of $[0,1]\times B^3$ in $W$ that agrees with $\iota^U_{0}$ on each $\{t\}\times B^3$ and agrees under the parametrization of the boundary by $Y_1$ with $\iota^U_1$. The same holds for $\iota^{\hash}$.
\end{itemize}
Morphisms between $(Y_0,\iota_0)$ and $(Y_1,\iota_1)$ consist of these pairs up to diffeomorphism commuting with the extra decorations.
\end{defn}
We will use the embedding $\iota^U$ to define the module structure while we will use $\iota^{\hash}$ to perform connected sums. Indeed the key point in the definition of a cobordism category is that the objects are three manifolds (not up to diffeomorphism). In all our construction we will hence make sure that all the three manifolds that arise come with natural identifications with the ones we are starting with.
\\
\par
We will only sketch the construction as we will prove a more involved result later from which this follows. Our approach, which involves families of metrics, is slightly different from that proposed in \cite{BMO}, which is based on moving basepoints. On the other hand, this alternative viewpoint is easier to generalize we are actually interested in.
\par
Consider the diffeomorphism $\varphi$ from $B^3\setminus B^3(1/2)$ to itself obtained by applying the antipodal map at each fix radius. Fix a three manifold $(Z,\jmath)$ there $\jmath$ is an embedding of $B^3$.
Given an object $(Y,\iota^U,\iota^{\hash})$ we can form their connected sum by first removing $\iota^{\hash}(B^3(1/2))$ and $\jmath(B^3(1/2))$ and then identifying $\iota(B^3\setminus B^3(1/2))$ and $\jmath(B^3\setminus B^3(1/2))$ using the diffeomorphism $\varphi$. This yields a manifold, which we simply denote by $Y\hash Z$ by dropping the notation, which is well defined, not only up to diffeomorphism.
\par
Given a based three manifold $(Z,\jmath)$, we can associate to it two Floer functors. First, we have the Floer functor $\mathscr{G}_Z$ associates to the based three manifold $(Y,\iota)$ the Floer group $\HMf_{\bullet}(Y\#Z)$. It is clear that $\mathscr{G}_Z$ maps surgery triples to exact triangles.
\\
\par
Before defining the Floer functor $\mathcal{F}_Z$ we need to review the construction of the $U$ action. Remove from $\mathbb{R}\times Y$ a ball and suppose that the metric near the new boundary component is a cylinder on the standard round one. Fix some additional small perturbation on the additional $S^3$ end so that no irreducible solutions are introduced. The chain complex $\hat{C}_{\bullet}(S^3)$ has trivial differential and is identified as a graded vector space with $\ztwo[[U]]\langle-1\rangle$. In particular the $U$ corresponds to the second negative eigenspace. We can then consider the moduli spaces of solutions to the Seiberg-Witten equations (in the blow-up) on $\mathbb{R}\times Y\setminus B^4$ that converge to $U$ on the additional $S^3$ end. These can be packed in order to define a chain map
\begin{equation*}
\hat{m}(U)_Y:\hat{C}_{\bullet}(Y)\rightarrow\hat{C}_{\bullet}(Y) 
\end{equation*}
that induces the action of $U$ at the homology level. In our case the embedding $\iota^U$ determines a preferred embedded ball obtained by applying the map $\iota^U$ fiberwise to
\begin{equation*}
B^4\subset[-1,1]\times B^3.
\end{equation*}
The same applies for $Z$, so that there is a chain map $\hat{m}(U)_Z$ inducing the action of $U$ at the homology level.
\\
\par
The functor $\mathscr{F}_Z$ is then defined to be the mapping cone of the chain map
\begin{equation*}
\left(1\otimes \hat{m}(U)_Z+\hat{m}(U)_Y\otimes 1\right):\hat{C}_{\bullet}(Y)\otimes\hat{C}_{\bullet}(Z)\rightarrow\left(\hat{C}_{\bullet}(Y)\otimes\hat{C}_{\bullet}(Z)\right)\langle 1\rangle.
\end{equation*}
The functor $\mathscr{F}_Z$ is actually independent in a natural way of the various choices, and one can actually show that it is a Floer functor (see Section \ref{formula}).

\vspace{0.5cm}
The key point in the proof is the construction of a natural transformation $\Psi$ from $\mathscr{F}_Z$ to $\mathscr{G}_Z$, which we now sketch. There is a natural cobordism $X_{\#}$ from $Y\amalg Z$ to the connected sum $Y\# Z$ given by attaching a one handle $[-1,1]\times B^3(3/4)$ to $[0,1]\times(Y\amalg Z)$ along the two balls in $\{1\}\times(Y\amalg Z)$ given by $\iota^{\hash}(B^3(3/4)))$ and $\jmath^{\hash}(B^3(3/4))$. After fixing appropriate corner roundings (independent of $Y$ and $Z$) the new boundary component is canonically identified with the connected sum $Y\#Z$.
\par
The four manifold $X^{\hash}$ contains a four ball $B^4(1/2)\subset [-1/2,1/2]\times B^3(1/2)$ in the one handle we added. As we are given standard local models, we can canonically define two hypersurfaces $Y'$ and $Z'$ identified with $Y$ and $Z$ so that $Y'$ separates $X^{\hash}\setminus B^4(1/2) $ in a copy of $I\times Y\setminus B^4$ and $X^{\hash}$, and similarly for $Z'$ (see the central picture in Figure \ref{Xhash}).
\\
\par
Fix the metric and pertubation $(g_0,\mathfrak{p}_0)$ on $X^{\#}$ so that near the boundary is a product with the on $Y$ and $Z$ we have already fixed in the discussion above. Recall that the manifold $(X_{\#})^*$ is obtained from it by attaching the cylindrical ends $(-\infty,-1]\times (Y\amalg Z)$ and $[0,\infty)\times Y\#Z$. We can choose this data so that the moduli spaces are regular. In particular counting solutions induces a chain map
\begin{equation}\label{hatm}
\hat{m}: \hat{C}_{\bullet}(Y)\otimes \hat{C}_{\bullet}(Z)\rightarrow \hat{C}_{\bullet}(Y\#Z).
\end{equation}
This is the main reason why we are working in the \textit{from} theory, as the analogous result in the \textit{to} theory is false.
\par
We then construct a family of metrics and perturbations $(g_1,\mathfrak{p}_1)$ on the cobordism $X^{\#}\setminus B^4$ parametrized by $(-\infty,\infty)$ with the following properties.
\begin{itemize}
\item For each $T$ they are a product near the boundary with the data we have already fixed for $Y,Z$ and $S^3$.
\item For $T\leq -2$, the manifold obtained from it by attaching cylindrical ends is isometric to the manifold obtained from $((X_{\#}), g_0)^*$ by removing the ball $B^4$ centered at $T$ from $(-\infty,-1]\times Y$. Here we are again using the parametrization given by $\iota^U$. Similarly for $T\geq 2$ we remove the ball centered at $T$ from $(-\infty,-1]\times Z$.
\end{itemize}
This family is depicted in Figure \ref{Xhash}. The existence of such a family follows from the fact that the space of metrics and perturbations is contractible. We will discuss in the next section how this family can be chosen so that all the moduli spaces are transverse (we will actually prove a much more general statement).
\begin{figure}
  \centering
\def\svgwidth{0.9\textwidth}
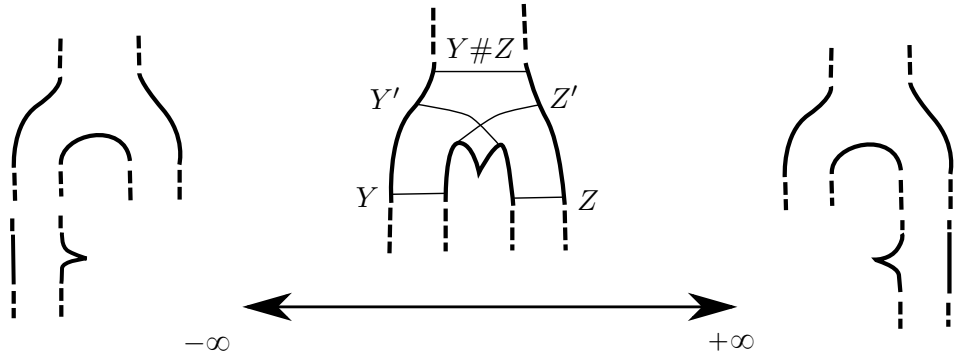
    \caption{The one parameter family of metrics we are considering on the cobordism $(W_{\#})^*$ with a ball removed. The two cobordism on the sides have the standard metric $g_0$.}
    \label{Xhash}
\end{figure} 

This moduli space is not compact as the parameter space is not. This is the usual phenomenon occurring when studying nect stretching argument, see for example Chapter $24$ of \cite{KM}. In particular suppose that we are given a family of solutions $\gamma_n$ corresponding to data $(g_1(T_n),\mathfrak{p}_1(T_n))$ which are asymptotic to $U$ on the $S^3$ end. If $T_n$ goes to $-\infty$, this converges (after passing to a subsequence) in some suitable sense to a pair consisting of
\begin{itemize}
\item a solution on the cylinder $\mathbb{R}\times Y$ with a ball removed that converges to $U$ on the $S^3$ end
\item a solution on $((W_{\#}), g_0)^*$.
\end{itemize}
One can think of this pair as a solution on the cobordism with the degenerate metric at $T=-\infty$, see Figure \ref{Xhash}. Using these maps and the same formulas as in the definition of $\hat{m}$ (Equation \ref{hatm}) one can define a map
\begin{equation*}
\hat{m}_U: \hat{C}_{\bullet}(Y)\otimes \hat{C}_{\bullet}(Z)\rightarrow \hat{C}_{\bullet}(Y\#Z)
\end{equation*}
and by considering the boundary strata one obtains the identity
\begin{equation*}
\hat{\partial}\circ \hat{m}_U+\hat{m}_U\circ \hat{\partial}=\hat{m}(U_Y\otimes 1+1\otimes U_Z).
\end{equation*}
In particular the map
\begin{equation*}
\hat{m}\oplus\hat{m}_U: \mathrm{Cone}\left( \hat{C}_{\bullet}(Y)\otimes \hat{C}_{\bullet}(Z)\stackrel{U_Y\otimes 1+1\otimes U_Z}{\longrightarrow} \hat{C}_{\bullet}(Y)\otimes \hat{C}_{\bullet}(Z)  \right)\rightarrow \hat{C}_{\bullet}(Y\#Z)
\end{equation*}
is a chain map, and we define the induced map in homology to be the natural transformation
\begin{equation*}
\eta_Z(Y): \mathscr{F}_Z(Y)\rightarrow \mathscr{G}_Z(Y).
\end{equation*}
Of course one has to check that this map is actually a natural transformation, and that the hypothesis of Theorem \ref{floerf} are satisfied. We will do this in Section \ref{formula} in the more complicated setting of $\Pin$-monopole Floer homology.

\vspace{1cm}
\section{A transversality result}\label{transversality}

In this section we discuss a technical result which is needed in the constructions. In the setting of the previous section, we have metrics and perturbations parametrized by the boundary of an interval $\partial[0,1]$ (where the boundary consists of degenerate metrics for which a neck is stretched to infinity) and we would like to extend it in the interior so that the usual transversality, compactness and gluing results hold for this parametrized family. Such a construction is implicitly used also in the proof of the surgery exact triangle, see \cite{KMOS} and \cite{Lin2}. We will always be referring to Chapters $24$ and $26$ of \cite{KM}.
\par
Before discussing the general result, we see how this construction works in the simplest case of composing cobordisms (see Chapter $26$) in the case of usual (non $\Pin$) monopole Floer homology. The main idea is expressed in Figure \ref{pert}.
\begin{figure}
  \centering
\def\svgwidth{0.9\textwidth}
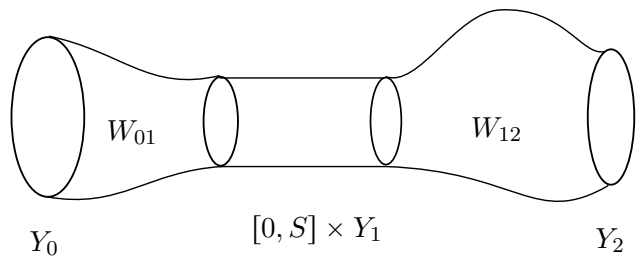
    \caption{In the construction of \cite{KM}, transversality is achieved by adding a perturbation on $W_{01}$ and $W_{12}$ independent of $S$. In our case, the perturbation on these two cobordisms is prescribed, and transversality is achieved by adding $S$-dependent perturbations on $[0,S]\times Y_1$.}
    \label{pert}
\end{figure} 
The construction there, which we now quickly review, does not satisfy our purposes as it is not the one we want on the boundary of the parametrizing interval. Suppose we are given cobordisms $W_{01}$ from $Y_0$ to $Y_1$ and $W_{12}$ from $Y_1$ to $Y_2$. Fix metrics so that the cobordisms are cylindrical near the ends, and regular tame perturbations $\p_0$, $\p_1$ and $\p_2$ on the three manifolds. For each $S\geq 2$, we can form the composite cobordism
\begin{equation*}
W(S)=W_{01}\cup ([0,S]\times Y_1)\cup W_{12},
\end{equation*} 
which is naturally equipped with a Riemannian metric. The perturbations $\p_i$ give rise to perturbations $\hat{\p}_i$ on the cylindrical parts of $W(S)^*$, the manifold obtained by adding cylindrical ends. The main transversality result is then Proposition $26.1.3$, which says that for generic choice of some additional compactly supported perturbation $\hat{\p}$ (independent of $S$) on the $W_{01}$ and $W_{12}$ the union over $S$ of the moduli spaces on $W(S)^*$ is transversely cut out. This readily follows from the fact that we can see the moduli space as a fibered product in which one of the factors is
\begin{equation}\label{parmod}
\mathcal{M}={\bigcup_{S\in[0,\infty)}} \{S\}\times M([0,S]\times Y_1)
\end{equation}
for which transversality is automatic (see Proposition $24.3.1$), and the proof of Proposition $24.4.7$. A family of these solutions with $S\rightarrow \infty$ will converge on each cobordism to a solution of the equations perturbed by $\hat{\p}$. In particular, we cannot use this approach to extend perturbations.
\par
Suppose then we have fixed perturbations $\hat{\p}_{01}$ and $\hat{\p}_{12}$ on the cobordisms so that the moduli spaces on $W_{01}^*$ and $W_{12}^*$ are transversely cut out. Consider a bump function $\beta(t)$ supported in $[0,1]$. Given an admissible perturbations $\mathfrak{q}_{\pm}$ on $Y_1$, we can add to our equations the perturbation
\begin{equation}\label{extrapert}
e^{-S}(\beta(t)\hat{\mathfrak{q}}_-+\beta(t+S-1)\hat{\mathfrak{q}}_+),
\end{equation}
which is supported on $[0,S]\times Y_1$. Of course these perturbations are mild enough to retain the nice compactness results (see Sections $24.5$ and $25.6$). Our claim is that the union of the moduli spaces over $W(S)^*$ is transversely cut out for generic choice of $\mathfrak{q}_{\pm}$. The proof of this claim is essentially identical. First, we form the parametrized moduli space $\mathcal{M}_{\mathfrak{q}_{\pm}}$ defined as Equation (\ref{parmod}) with our extra $S$ dependent perturbations (\ref{extrapert}). This has the natural structure of a Banach manifold (the proof of Proposition $24.3.1$ applies verbatim). We can also form the Banach manifold
\begin{equation*}
\mathfrak{M}=\bigcup_{\mathfrak{q}_{\pm}\in \mathcal{P}(Y_1)\times \mathcal{P}(Y_1)} \mathcal{M}_{\mathfrak{q}_{\pm}}
\end{equation*}
parametrized by two copies of the space of tame perturbations $\mathcal{P}(Y_1)$.
The key result is that (using the appropriate Sobolev completions) the linearization of the restriction map
\begin{equation*}
R:\mathfrak{M}\rightarrow \mathcal{B}^{\sigma}(\bar{Y}_1\amalg Y_1)
\end{equation*}
has range dense in the $L^2_{1/2}$-topology at a point $([\gamma],\mathfrak{q})$ with $[\gamma]$ irreducible, and similarly 
\begin{equation*}
R:\mathfrak{M}^{\mathrm{red}}\rightarrow \partial \mathcal{B}^{\sigma}(\bar{Y}_1\amalg Y_1)
\end{equation*}
has range dense in the $L^2_{1/2}$-topology if $[\gamma]$ is reducible. Here, following the notation of \cite{KM}, $\mathfrak{M}^{\mathrm{red}}$ denotes the reducible part of the moduli space. Indeed we have a stronger result, which is Lemma $24.4.8$: the linearization of the restriction map on the moduli space on $[0,S]\times Y_1$ has dense image for each $S$. From this, the desired transversality result follows from the usual fibered product description, see the proof of Proposition $24.4.7$.
\\
\par
In the $\Pin$ case some extra care has to be taken, as one might not be able to achieve transversality by only using tame perturbations. In \cite{Lin} (see Section $4.2$) a more general class of perturbations, called $\Pin$-equivariant \textsc{asd}-perturbations, is introduced. While tame perturbations are inherently three dimensional, this new type of perturbations is defined on the blown-up configuration space a cobordism. It is shown in \cite{Lin} that this class is large enough to achieve transversality for the moduli spaces on cobordisms (while retaining the nice analytical properties). In our setting, the same discussion above implies that for a generic choice of $\hat{\omega}_{\pm}\in\mathcal{P}_{\textsc{asd}}(I\times Y_1)$, after adding the $S$-dependent perturbation
\begin{equation*}
e^{-S}(\hat{\omega}_-+\tau_{S-1}^*\hat{\omega}_+)
\end{equation*}
the parametrized moduli spaces are transversely cut out. Here $\tau_{S-1}^*\hat{\omega}_+$ denotes the translation of $\hat{\omega}_+$ so that it is a perturbation in $[S-1,S]\times Y_1$, so that the extra perturbation is supported on $[0,S]\times Y_1$.
\\
\par
We now discuss the general case of a family of metrics and perturbations on a four manifold $X$ parametrized by $P$ a compact smooth manifold with corners (with some additional data). In our specific case, $P$ will be an associahedron or a multiplihedron (see the next section).
\begin{defn}\label{standbound}
We say that a family of (possibly degenerate) metrics and perturbations parametrized by $P$ is \textit{standard at the boundary} if the following conditions hold.
\begin{enumerate}
\item There is a correspondence between the faces $\Delta_i$ of $P$ and a collection of embedded hypersurfaces $Y_i\subset X$. Two faces intersect if and only if the corresponding hypersurfaces are disjoint.
\item For each face $\Delta_i$, there is a neighborhood $U_i$ in $P$ identified with $(T_i,\infty]\times \Delta_i$ such that given a collection of faces $\Delta_1,\dots,\Delta_k$ we have a natural identification
\begin{equation*}
\bigcap_{i=1}^k U_i\equiv \prod_{i=1}^k(T_i,\infty]\times \bigcap_{i=1}^k\Delta_i
\end{equation*}
\item If $p=(S_i, q)\in U_i$, the corresponding metric is obtained by attaching the cylinder $[0,S_i]\times Y_i$ to the manifold with cylindrical ends corresponding to $q$ (which has a cylindrical end modeled on $\left([0,+\infty)\amalg(-\infty,0]\right)\times Y_i$).
\end{enumerate}
We call $\{U_i\}$ a \textit{compatible system of boundary neighborhoods}, and think of it as part of the data associated to the manifold with corners $P$.
\end{defn}
For example, in the previous discussion $P$ was a closed interval. In the case of the surgery exact triangle, $P$ is obtained by identifying the edges of five squares, so that it is a pentagon. The conditions imply that in the neighborhood of a codimension $k$ facet the metric is obtained by stretching along $k$ disjoint hypersurfaces. Notice that $\{U_i\}$  induces a compatible system of boundary neighborhoods on each facet $\Delta$, and the restriction of the family to $\Delta$ is standard at the boundary. In our case, the existence of such a compatible system of boundary neighborhoods will follow from a nice description of the associahedron as a union of cubes (see Lemma \ref{associaneigh}).
Our main transversality result is the following.

\begin{prop}\label{fundtr}
Suppose we are given a smooth compact manifold with corners $P$ with a compatible system of boundary neighborhoods $\{U_i\}$, and a regular family of metrics and perturbations on $\partial P$ such that the restriction on each codimension one facet is standard at the boundary. Suppose in addition that the family metrics is compatible, i.e. one can extend it to a neighborhood $U$ of $\partial P$ so that it is standard at the boundary. Then there is a regular family of metrics and perturbations on $P$ which is standard at the boundary and extends the given one on $\partial P$. The space of such extensions is contractible.
\end{prop}

\begin{proof}
Consider first the usual (non $\Pin$) case. Pick any smooth extension to the whole $P$ which is standard at the boundary. The compactness properties for this parametrized moduli spaces are clear (see again Chapter $26$ of \cite{KM}), so we just need to focus on transversality. Consider the neighborhood $U_i\equiv (T_i,\infty]\times \Delta_i$ of $\Delta_i$ in $P$. Choose a function $\varphi_i$ on $(T_i,\infty]$ which is zero near $T_i$ and $e^{-S_i}$ for very large $S_i$, and a function $\psi_i$ on $\Delta$ which is always positive and decays as $e^{-S_j}$ near the boundary component $\Delta_i\cap \Delta_j$ (such a function can be easily written down using the compatible system of boundary neighborhoods induced on $\Delta$). For each $i=1,\dots, k$ fix perturbations $\mathfrak{q}_{i,\pm}$ in $\mathcal{P}(Y_i)$. For $(S_i,q)\in U_i$ we can add as in equation (\ref{extrapert}) the perturbations
\begin{equation*}
\varphi_i(S_i)\psi_i(q)\left(\beta(t)\hat{\mathfrak{q}}_{i,-}+\beta(t+S_i-1)\hat{\mathfrak{q}}_{i,+}\right),
\end{equation*}
each of which is supported in (a collar of the boundary of) $[0,S_i]\times Y_i$ and vanishes at $\partial P$. The same argument as above implies that for generic choice of the $\mathfrak{q}_{i,\pm}$ this satisfies the hypothesis. The extension to the interior of $P$ then follows from the usual extension result from a closed subset, Proposition $24.4.10$. Finally, the contractibility follows from the contractibility of the space of metrics, perturbations, and the space of functions we have chosen.
\par
The result in the $\Pin$ case can be achieved in the same spirit as the previous discussions by adding extra perturbations supported in $[0,S_i]\times Y_i$ of the form
\begin{equation*}
\varphi_i(S_i)\psi_i(q)\left(\hat{\omega}_{i,-}+\tau_{S_i-1}^*\hat{\omega}_{i,+}\right),
\end{equation*}
for generic $\hat{\omega}_{i,\pm}\in\mathcal{P}_{\textsc{asd}}(I\times Y_i)$.
\end{proof}

\vspace{1cm}
\section{Higher compositions}\label{higher}

In this section we prove Theorem \ref{ainf}. We begin by recalling how the usual module structure is defined in our context. We have briefly discussed this in section \ref{HMconn} but we will review it here in more detail (also because we need to deal with more intricate transversality issues).
\par
Consider the cobordism $I\times Y$ with an open ball removed, seen as a cobordism with incoming ends $Y$ and $S^3$ and outgoing end $Y$, one obtains the map
\begin{equation}\label{modact}
\HSf_{\bullet}(Y)\otimes \HSf_{\bullet}(S^3)\rightarrow \HSf_{\bullet}(Y),
\end{equation}
and we can identify $\HSf_{\bullet}(S^3)$ with $\Rin$ (after a grading shift). More explicitly, the map is defined as follows. Suppose we have fixed on $S^3$ a metric of positive scalar curvature. It follows that for small perturbations there are no irreducible critical points and only one reducible solution. In particular the chain complex computing $\HSf_{\bullet}(S^3)$ is the direct sum of the chain complexes of the unstable critical submanifolds $[\mathfrak{C}_{i}]$, $i\leq -1$, and the homology is indeed isomorphic to $\Rin$. To define the action of a class $r\in \Rin$  we can then proceed as follows. Choose on $(I\times Y)\setminus \mathrm{int}B^4$ a metric which is a product with a metric of positive scalar curvature on the incoming $S^3$ end. Choose a smooth chain $\Delta$ representing the class $r$. For a generic choice of perturbations on the cobordism with cylindrical ends attached, we can suppose that the moduli spaces are transverse to $\Delta$, so that we can take the fibered products. Thus we can define the maps
\begin{equation*}
m(\Delta)^o_o: C^o(Y)\rightarrow C^o(Y)
\end{equation*}
and the seven companions $m(\Delta)^o_s,m(\Delta)^u_o,m(\Delta)^s_u, \bar{m}(\Delta)^s_s,\bar{m}(\Delta)^s_u,\bar{m}(\Delta)^u_u$ and $\bar{m}(\Delta)^u_s$. We can use these to define the map
\begin{equation*}
\hat{m}(\Delta):\hat{C}_{\bullet}(Y)\rightarrow \hat{C}_{\bullet}(Y)
\end{equation*}
by the formula
\begin{equation}\label{chainmod}
\hat{m}(\Delta)=
\begin{bmatrix}
m^o_o & m^u_o \\
\bar{m}^s_u\partial^o_s+\bar{\partial}^s_u m^o_s & \bar{m}^u_u+\bar{m}^s_u\partial^u_s+\bar{\partial}^s_u m^u_s
\end{bmatrix}.
\end{equation}
Recall that $\hat{C}_{\bullet}(Y)$ is the direct sum $C^o(Y)\oplus C^u(Y)$. One can check that this is actually a chain map (using the fact that there are no irreducible solutions on $S^3$). We define the induced map to be the action of $r$ on $\HSf_{\bullet}(Y)$.
\par
We can also see the map (\ref{modact}) from an alternative viewpoint which will generalize to higher compositions. Consider the subcomplex
\begin{equation*}
\hat{C}_{\bullet}(Y,2)\subset \hat{C}_{\bullet}(Y)\otimes \hat{C}_{\bullet}(S^3)
\end{equation*} 
generated by chains $\sigma\otimes\tau$ such that $\sigma\times \tau$ is is transverse to the moduli space on $((I\times Y)\setminus \mathrm{int}B^4)^*$. The inclusion of this chain complex is a quasi-isomorphism, as it can be seen by using the energy filtration (see Section $3.2$ of \cite{Lin}). The fibered products give rise to a chain map
\begin{equation*}
m_2:\hat{C}_{\bullet}(Y,2)\rightarrow \hat{C}_{\bullet}(Y)
\end{equation*}
inducing the module structure (\ref{modact}).  
\\
\par
We now generalize the construction we have just described to construct the higher compositions
\begin{align*}
\mu_n&: \hat{C}_{\bullet}(S^3,n)\rightarrow \hat{C}_{\bullet}(Y)\\
m_n&: \hat{C}_{\bullet}(Y, n)\rightarrow \hat{C}_{\bullet}(Y)
\end{align*}
where
\begin{align*}
\hat{C}_{\bullet}(S^3,n)&\subset\hat{C}_{\bullet}(S^3)^{\otimes n}\\
 \hat{C}_{\bullet}(Y, n)&\subset \hat{C}_{\bullet}(Y)\otimes \hat{C}_{\bullet}(S^3)^{\otimes n-1}
\end{align*}
are chain complexes such that the inclusions are quasi-isomorphisms.
\\
\par
As remarked in the previous section, a key point is that the constructions should be made so that the three manifolds arising come with canonical identifications with those we are interested in. To do this, we will sometime require some additional data to be fixed.
\par
From the cobordism $I\times Y$ we can remove $d-1$ open balls in order to obtain a cobordism $W_d$ with incoming ends $Y, S^3_1,\cdots, S^3_{d-1}$. As in Section \ref{HMconn}, these balls are parametrized (after a suitable shifting in the $I$ component and linear rescaling, both independent of $Y$) by applying to
\begin{equation}\label{4ball}
B^4\subset [-1,1]\times B^3
\end{equation}
the map $\iota^U$ fiberwise. It will be convenient to consider only the image of $B^4(1/2)$. Here we label the incoming $S^3$ ends by $S^3_i$ and we consider the set of incoming ends as an ordered set. To simplify the notation we will sometime denote $Y$ by $S^3_0$. Similarly denote by $Z_d$ the cobordism $I\times S^3$ with $d-1$ balls removed, so that it has $d$ incoming ends and one outgoing end.
\par
On $W_d$, for each pair of incoming ends $S^3_i$ and $S^3_j$ with $i<j$ we can define a hypersurface $\Sigma_{i,j}$ schematically depicted in Figure \ref{Wd}. This is canonically identified of $S^3$ if $i>0$, while it is identified of $Y$ for $i=0$. The hypersurface $\Sigma_{i,j}$ separates the cobordism $W_d$ in two pieces:
\begin{itemize}
\item if $i=0$, two cobordisms which are diffeomorphic to $W_j$ and $W_{d-1-j}$;
\item if $i=1$, two cobordisms, one diffeomorphic to $W_{d-2+(j-i)}$ and one to $Z_{i-j+1}$.
\end{itemize}
Two hypersurfaces $\Sigma_{i,j}$ and $\Sigma_{k,l}$ intersect if and only if the intervals $\{i,,\dots, j\}$ and $\{k,\dots,l\}$ intersect and are not contained one in the other. When the intersection is not empty it is diffeomorphic to $S^2$.
\\
\par
We can specify the hypersurfaces as follows. The hypersurface $\Sigma_{0,j}$ is an hypersurface of the form $\{(t_j)\}\times Y$. The hypersurfaces identified with $S^3$ come the map $\iota^U$ applied fiberwise to a ball $B^4(R)$ as in (\ref{4ball}) after shifting in the $I$ direction and linearly rescaling (using a fixed model independent of $Y$). We can choose the parameters so that the intersections are as described (any two choices are diffeomorphic). This is shown schematically in Figure \ref{Wd}. We will depict them more conveniently as in the left of Figure \ref{K4}, where the cusps denote the incoming $S^3$ ends. 
Again, the key point here is that these hypersurfaces come with canonical diffeomorphisms with respectively $Y$ and $S^3$.
\begin{figure}
  \centering
\def\svgwidth{\textwidth}
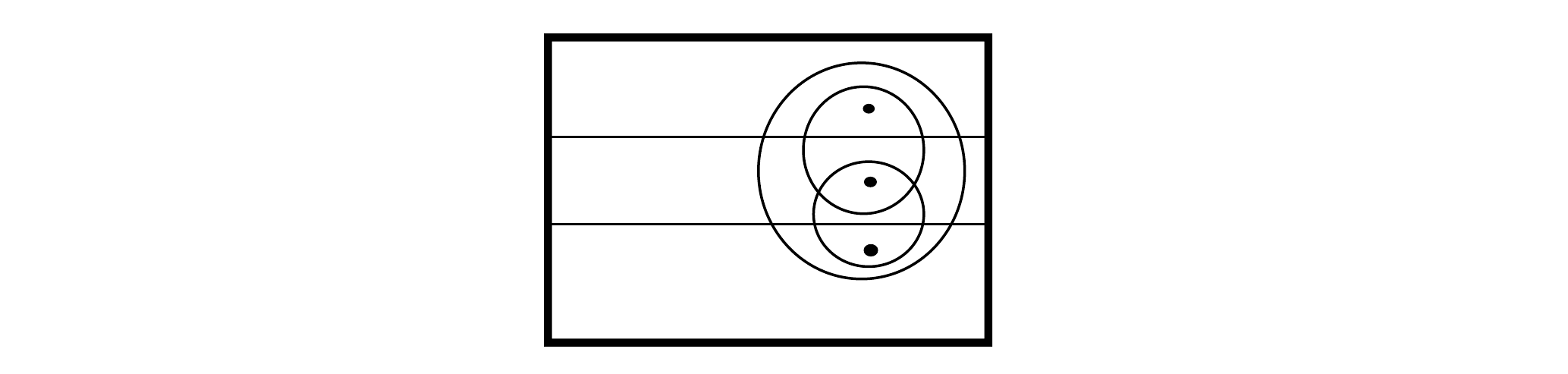
    \caption{The five separating hypersurfaces in $I\times Y$ with three balls removed.}
    \label{Wd}
\end{figure} 

\vspace{0.3cm}

Using these hypersurfaces, we define a family of metrics parametrized by the $n$-\textit{associahedron} $K_n$. Recall that this is an $(n-2)$-dimensional convex polytope that plays a key role in the context of $\Ainf$-structures. The face poset of the associahedron corresponds to the coherent parenthesizations of the ordered set $\{0,\dots, n-1\}$. The key property is that the codimension one faces of the associahedron $K_n$ are naturally identified with products of smaller dimensional associahedra $K_j\times K_{n-1-j}$. The associahedron $K_3$ is the interval, while the associahedron $K_4$ is the pentagon.
\par
In our setting the best way to think about the associahedron $K_n$ is the following. Each codimension one face correspond to a single parenthesis to which we can associate a single hypersurface $\Sigma_{ij}$ in our cobordism $W_n$ or $Z_n$. We will think of a face as parametrizing a family of metrics on which the corresponding hypersurface is stretched to infinity. This can be done consistently because two faces that intersect correspond to disjoint hypersurfaces. To inductively apply the transversality result, Proposition \ref{fundtr}, we need the following lemma.
\begin{lemma}\label{associaneigh}
The associahedron $K_n$ admits a compatible system of boundary neighborhoods (in the sense of Definition \ref{standbound}). The restriction fn the compatible system to a facet $K_j\times K_{n-1-j}$ is the product of the compatible systems on each factor.
\end{lemma}
\begin{proof}
There is an alternative description of $K_n$ that makes the result clear. To each $n$-uple of non-intersecting hypersurfaces $\{Y_1,\dots, Y_n\}$ associate the cube $[0,\infty]^n$, where each hypersurface $Y_i$ corresponds to the coordinate $t_i$. Given two $n$-uples of hypersurfaces which differ only by one element, we identify the faces of the respective cubes in which the coordinate corresponding to the different hypersurface is zero. The result is the associahedron $K_n$, and the neighborhood $U_i$ (corresponding to the hypersurface $Y_i$) obtained by taking the union of the open sets $\{t_i>1\}$ over all cubes in which $Y_i$ appears.
\end{proof}

We now describe explicitly how the construction is performed order to define the maps $\mu_n$ and $m_n$ satisfying the appropriate composition laws. As the discussion above suggests, we will take an inductive approach.
\par
Fix any regular metric and perturbation $(g_1,\p_1)$ and $(h_1,\mathfrak{q}_1)$ on $Y$ and $S^3$, with the additional requirement that for the latter the metric has positive scalar curvature and the perturbation is small so that there are no reducible solutions. From this we can define the maps $m_1$ and $\mu_1$, which are just the differentials of the chain complexes $\hat{C}^{\jmath}_{\bullet}(Y)$ and $\hat{C}^{\jmath}_{\bullet}(S^3)$. The maps $\mu_2$ and $m_2$ are just the chain maps giving rise to the module structures described above. In particular for $m_2$ we consider the manifold $W_2$ and a metric and a pertubation $(g_2,\p_2)$ on it which coincide with the given ones for $Y$ and $S^3$ in a neighborhood of the boundary and for which the moduli spaces are regular. Similarly we have fixed a pair $(h_2,\mathfrak{q}_2)$.
\\
\par
The first non trivial case is the family of metrics and perturbations on $W_3$, which is analogous to the construction in the precious section. In this case the family is parametrized by the associahedron $K_3$, which is identified with the interval $[-\infty,\infty]$. The hypersurface $\Sigma_{01}$ cuts the cobordism in two copies of $W_2$, so that we can associate to the point $-\infty$ two copies pair $(g_2,\p_2)$. Similarly the hypersurface $\Sigma_{12}$ cuts the cobordism in a copy of $W_2$ and a copy of $Z_2$, so that over the point $+\infty$ we can associate the pair which is $(g_2,\p_2)$ on $W_2$ and $(h_2,\mathfrak{q}_2)$ on $Z_2$. We can then apply Proposition \ref{fundtr} to obtain a family $(g_3,\p_3)$ on $W_3$ parametrized by the closed interval so that for $T$ negative enough it is (up to an exponentially decaying pertubation term) obtained by attaching cylinder of length $T$ to the outgoing end of the first copy of $W_2$ and the incoming end of second copy of $W_2$, and similarly for $T$ positive enough, see Figure \ref{stretch}. In an analogous manner one can define a family of metrics and pertubations $(h_3,\mathfrak{q}_3)$ on $Z_3$.
\begin{figure}
  \centering
\def\svgwidth{\textwidth}
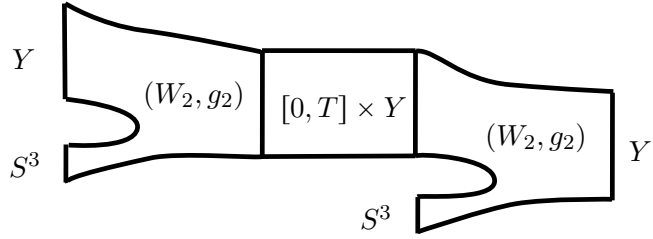
    \caption{The family of metrics $g_3$ for $T$ negative.}
    \label{stretch}
\end{figure} 

The definition of the family of pairs $(g_n,\p_n)$ and $(h_n,\mathfrak{q}_n)$ parametrized by the associahedron $K_n$ now proceeds inductively. We consider the case when $n$ is equal to $4$. In this case $K_4$ is a pentagon (obtained by gluing together $5$ squares) and each edge corresponds to a hypersurface $\Sigma_{ij}$. Two edges share a vertex if and only if the corresponding hypersurfaces are disjoint. Each of these hypersurfaces separates $W_4$ in a pair of cobordisms on each of which we have already defined a family of metrics and perturbations, see Figure \ref{K4}. In particular we can use products of the families of metrics and perturbations we have already defined as the metrics parametrized by the edges of $K_4$. This family parametrized by the boundary can be extended in the interior to a regular family using Proposition \ref{fundtr}. The key point is that two codimension one faces of the associahedron intersect if and only if they correspond to disjoint hypersurfaces.
\begin{figure}
  \centering
\def\svgwidth{0.9\textwidth}
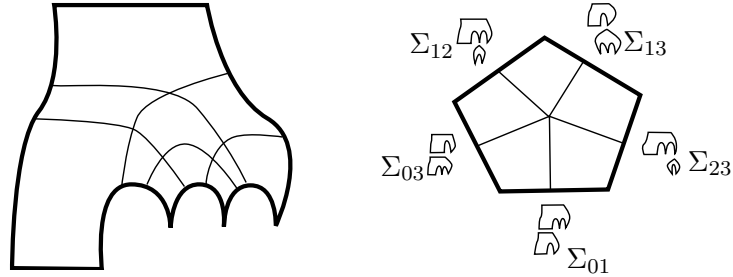
    \caption{The two parameter family of data paramerized by $K_4$ that looks like stretching along the hypersurfaces $\Sigma_{ij}$. Each edge of $K_4$ corresponds to one of the hypersurfaces.}
    \label{K4}
\end{figure} 

The same idea then applies to each $n$. Each hypersurface divides the cobordism $W_n$  in a pair of cobordisms on which we have already defined a family of metric parametrized by the associahedra $K_j$ and $K_{n-j-1}$ for some $j$. This hypersurface corresponds to a face of the associahedron $K_n$ which is identified with the product $K_j\times K_{n-j-1}$. So the previously defined families induce one on the boundary of $K_n$, which we can extend inside thanks to Proposition \ref{fundtr}.
\\
\par
Our construction naturally leads to the following definition. 
\begin{defn}We say that a sequence $\mathscr{D}=\{(g_n,\p_n), (h_n,\mathfrak{q}_n)\}$ of metrics and pertubations on $W_n$ and $Z_n$ parametrized by the associahedron $K_n$ forms \textit{admissible data} if 
\begin{itemize}
\item all parametrized moduli spaces are regular;
\item for each $n$ the restriction of $(g_n,\p_n)$ to a codimension one face is given by the product of the corresponding families on the two cobordisms obtained by cutting along the corresponding hypersurface, and similarly for $(h_n,\mathfrak{q}_n)$.
\end{itemize}
\end{defn}
Having defined the admissible data $\mathscr{D}$ ready to define the higher compositions $m_n$ and $\mu_n$. Consider the subspace
\begin{equation*}
\Cj(Y,n)\subset \Cj(Y)\otimes \Cj(S^3)^{\otimes n-1}
\end{equation*}
generated by products $\sigma\otimes \tau_1\otimes\cdots \otimes \tau_{n-1}$ such that:
\begin{itemize}
\item for each $j=1,\dots n-1$ the chain $\sigma\times \tau_1\times\cdots \tau_j$ is transverse to all the compactified moduli spaces of solutions on $W_{j+1}$ parametrized by the family of metrics $(g_{j+1},\p_{j+1})$.
\item for each $1\leq j< l\leq n-1$ the chain $\tau_j\times\cdots \tau_l$ is transverse to all the compactified module spaces of solutions on $Z_{l-j+1}$ parametrized by the family of metrics $(h_{l-j+1},\mathfrak{q}_{l-j+1})$.
\end{itemize}
This is a subcomplex and the inclusion induces an isomorphism in homology, as it can be seen by looking at the spectral sequence associated to the energy filtration. 
\par
By using the same formulas as in (\ref{chainmod}) we can then construct the map
\begin{equation*}
m_n: \Cj(Y,S^3,n)\rightarrow \Cj(Y)
\end{equation*}
by taking fibered products with the moduli spaces.
The analogous definition gives rise (by substituting $Y=S^3$) to the maps
\begin{equation*}
\mu_n: \Cj(S^3,n)\rightarrow \Cj(S^3).
\end{equation*}
We will denote the whole package of chain complexes and higher compositions by $\Cj(Y,\mathscr{D})$ and $\Cj(S^3,\mathscr{D})$. We have the following key result.
\begin{prop}
The maps $m_n$ and $\mu_n$ satisfies the relations (\ref{infalg}) and (\ref{infmod}).
\end{prop}
This result is saying at a practical level that $\Cj(S^3)$ is an $\Ainf$-algebra and $\Cj(Y)$ is an $\Ainf$-module over it. We cannot use that terminology though because the compositions are not defined for all tuples of elements but only on those which satisfy the transversality conditions specified above.

\begin{proof}
First of all, the formulas make sense because of the transversality conditions we are requiring on the products of consecutive elements in our chains. The verification of the identities (\ref{infalg}) and (\ref{infmod}) then boils down to identifying the boundary strata of a fibered product of the form
\begin{equation*}
(\sigma\times\tau_1\times\cdots\times \tau_{n-1})\times M^+_{K_n}
\end{equation*}
where $M^+_{K_n}$ is any compactified moduli space of solutions parametrized by the family $K_n$ on the cobordism $W_n$. In particular we want to identify the component that lies in the singular chain complex of a given critical submanifold $[\Cr_+]$ in the outgoing end. Notice that $M^+_{K_n}$ has two kinds of boundary strata: the fibered products of moduli spaces of unparametrized trajectories of the various ends and those on $W_n$ parametrized by the family $K_n$, and strata parametrized by the faces of the associahedron $K_n$, which we denote $M^+_{\partial K_n}$
The fibered products of the first kind, together with the terms of the form
\begin{equation*}
\partial\left( (\sigma\times\tau_1\times\cdots\times \tau_{n-1})\times M^+_n\right),\qquad \left(\partial(\sigma\times\tau_1\times\cdots\times \tau_{n-1})\right)\times M^+_n
\end{equation*}
correspond to the $n+1$ terms in the formulas (\ref{infalg}) and (\ref{infmod}) involving the differentials $m_1$ or $\mu_1$. So it remains to consider the fibered products
\begin{equation*}
(\sigma\times\tau_1\times\cdots\times \tau_{n-1})\times M^+_{\partial K_n}.
\end{equation*}
By definition the families of metrics and perturbations on the boundary are products of those defined on $W_j$ by $(g_j,\p_j)$ and on $Z_k$ by $(h_k,\mathfrak{q}_k)$, so it is straightforward to identify these terms with those involving compositions $\mu_j$ and $m_k$ with $j,k\geq 2$ in the identities we want to prove.
\end{proof}
\vspace{0.5cm}

Having defined the higher compositions, we discuss in detail in which sense these are invariants of the manifolds, as stated vaguely in Theorem \ref{ainf}. If the objects involved were genuine $\Ainf$-algebras and modules, the result would say that for different choices of admissible data $\mathscr{D}$ and $\mathscr{D}'$, we can construct a weak homotopy equivalence of $\Ainf$-algebras
\begin{equation*}
h(\mathscr{E}):\Cj(S^3,\D)\rightarrow\Cj(S^3,\D')
\end{equation*}
depending of some additional data $\mathscr{E}$, and any two such weak homotopy equivalences we construct are homotopic. The similar statement holds for the $\Ainf$-modules $\Cj(Y,\mathscr{D})$.
\par
In our case we have to deal with some transversality constraint so we will define, under some suitable choice of a sequence of $\mathscr{E}$ of metrics and perturbations $(k_n,\mathfrak{r}_n)$ on the cobordisms $Z_n$, maps
\begin{equation}\label{fn}
f_n: \hat{C}_{\bullet}(S^3,n,\mathscr{E})\rightarrow \hat{C}_{\bullet}(S^3,(g'_1,\mathfrak{p}'_1))
\end{equation}
satisfying the relations of an $\Ainf$-morphism (\ref{amorp}) where
\begin{equation*}
\hat{C}_{\bullet}(S^3,n,\mathscr{E})\subset\hat{C}_{\bullet}(S^3,n,\D)
\end{equation*}
is a subcomplex so that the inclusion is a quasi-isomorphism.
\par
We start by defining the data $\mathscr{E}$, and it should not be surprising that the family of metrics and perturbations on the cobordism $Z_n$ will be parametrized by the \textit{multiplihedron} $J_n$. Recall that this is a $n-1$ dimensional polyhedron whose vertices correspond to the ways of put complete parenthesis and applying function $f$ to the string $\{0,\dots, n-1\}$. For example $J_2$ is an interval whose vertices correspond to the expressions $f(01)$ and $f(0)f(1)$, while $J_3$ is the hexagon shown in Figure \ref{J3}. The faces of the multiplihedron $J_n$ are of two kinds. The first are those parametrized by
\begin{equation*}
J_{i_1}\times \cdots \times J_{i_j}\times K_j\rightarrow J_n
\end{equation*}
for a partition $i_1+\cdots i_j=n$, and those parametrized by
\begin{equation*}
J_{n-e+1}\times K_e\rightarrow J_n
\end{equation*}
for $2\leq e\leq n$. In Figure \ref{J3}, the horizontal edge are of the second kind, while the others are of the first kind.
\begin{figure}
  \centering
\def\svgwidth{0.9\textwidth}
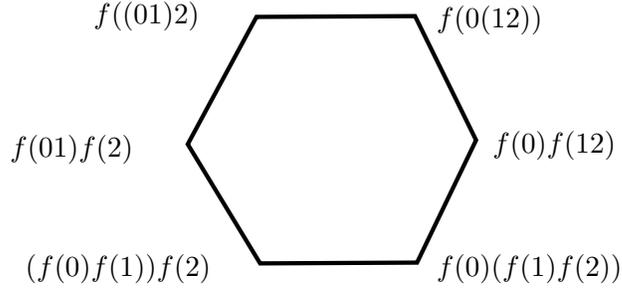
    \caption{The multiplihedron $J_3$.}
    \label{J3}
\end{figure} 
\\
\par
As usual, we define the data inductively. On the cobordism $Z_1$, which is just topogically a product $I\times Y$, we choose a metric and perturbation $(k_1,\mathfrak{r}_1)$ which agrees with the cylinders on the two given ones $(h_1,\mathfrak{q}_1)$ and $(h_1',\mathfrak{q}_1')$ near the boundary. As usual, we can choose them so that all the moduli spaces on the manifold obtained by adding cylindrical ends are regular. This choice induces a continuation map, which is denoted in our notation by $f_1$.
\par
On the cobordism $Z_2$ the family of metrics and perturbations parametrized by $J_2=[-\infty,\infty]$ all of which agree with $(h_1,\mathfrak{q}_1)$ on the incoming end and with $(h_1',\mathfrak{q}_1')$ on the outgoing end is defined as follows. The data at $-\infty$ is the one for which the incoming ends are stretched at infinity, so that the cobordism is decomposed as a disjoint union
\begin{equation*}
(Z_1\amalg Z_1)\amalg Z_2.
\end{equation*}
On the two copies of $Z_1$ we choose the data $(k_1,\mathfrak{r}_1)$ defined right above, while on $Z_2$ we choose the data $(h_2',\mathfrak{q}_2')$ defining the composition $\mu'_2$. At the other end $\infty$, we stretch the outgoing end to infinity so that the cobordism is decomposed as a union
\begin{equation*}
Z_2\amalg Z_1.
\end{equation*} 
On the first cobordism we consider the data $(h_2,\mathfrak{q}_2)$ defining the composition $\mu_2$ while on the latter we consider the data $(k_1,\mathfrak{r}_1)$ defined above. Using Proposition \ref{fundtr} can then extend the family as in the definition of the admissible data $\mathscr{D}$ to the multiplihedron $J_2$, so that when the parameter $T$ is negative enough the family is exponentially close to the one obtained by attaching the the two copies of $(Z_1,(k_1,\mathfrak{r}_1))$ and $(Z_2,(h_2',\mathfrak{q}_2'))$ with a cylinder $[0,T]\times Y$ with the product metric and perturbation with $(g_1', \p_1')$ and the similar thing for $T$ very positive, see Figure \ref{ainf22}.
\begin{figure}
  \centering
\def\svgwidth{0.9\textwidth}
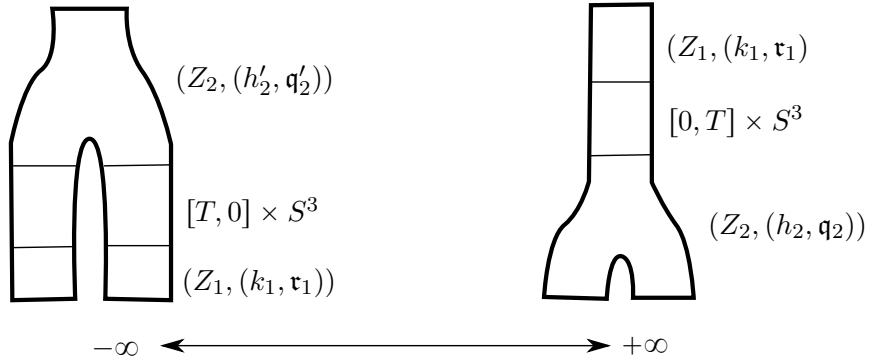
    \caption{The family of metrics and perturbations parametrized by the multiplihedron $J_2$. On the products $I\times S^3$ the metric and perturbations are the ones induced by $(h_1',\mathfrak{q}')$ on the three manifold.}
    \label{ainf22}
\end{figure} 
We can proceed inductively in the same way and define the family of data on $Z_n$ parametrized by $J_n$ as its facets can be identified with products of smaller dimensional associahedra and multiplihedra, and the data has been defined on them. The key point is that each face of the multiplihedron correspond to a (possibly disconnected) hypersurface, and two faces intersect if and only if they correspond to disjoint hypersurfaces, see Figure \ref{nearcorner}.
\begin{figure}
  \centering
\def\svgwidth{0.9\textwidth}
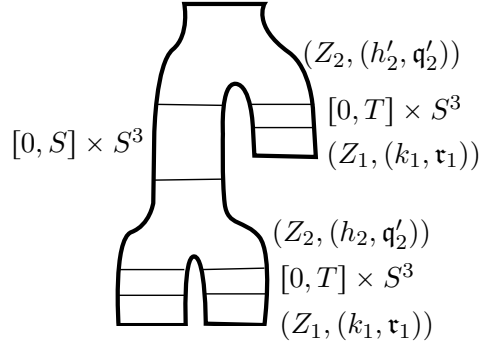
    \caption{The two parameter family parametrized by $T$ and $S$ near the corner $(f(0)f(1))f(2)$ of the associahedron $J_3$.}
    \label{nearcorner}
\end{figure}
\\
\par
We define the subcomplex $\Cj(S^3,n,\mathscr{E})$ as the subcomplex generated by all tensor products $\sigma_1\otimes \cdots \otimes \sigma_n$ so that each product of $k$ consecutive is transverse to all the parametrized moduli spaces on $Z_k^*$. The maps $f_n$ in Equation (\ref{fn}) are defined by taking fibered products with the moduli spaces. We have then the following invariance result.

\begin{thm}\label{invariance}
For a given family of data $\mathscr{E}$ constructed as above the maps $f_n$ satisfy the relations (\ref{amorp}) of an $\Ainf$-morphism with the higher compositions $\mu_n$ and $\mu_n'$. Given two choices of data $\mathscr{E}$ and $\mathscr{E}'$, there is a family of chain complexes
\begin{equation*}
\hat{C}_{\bullet}(S^3,n,\mathscr{E},\mathscr{E}')\subset \hat{C}_{\bullet}(S^3,n,\mathscr{E})\cap \hat{C}_{\bullet}(S^3,n,\mathscr{E}')
\end{equation*}
so that the inclusion induces a quasi-isomophism and there are maps
\begin{equation*}
h_n:\hat{C}_{\bullet}(S^3,n,\mathscr{E},\mathscr{E}')\rightarrow \hat{C}_{\bullet}(S^3,(g_1',\mathfrak{q}_1'))
\end{equation*}
satisfying the relations (\ref{ahom}) of an $\Ainf$ homotopy between the two $\Ainf$-morphisms $f$ and $f'$. In particular, the map induced in homology by $f_1$ is a well defined isomorphism.
\par
The analogous statement holds for the modules $\Cj(Y,\mathscr{D})$.
\end{thm}
\begin{proof}
The relations of an $\Ainf$-morphism follow by identifying the contributions given by the boundary strata of the multiplihedron. For the independence part, one can construct a family of data $(l_n,\mathfrak{s}_n)$ parametrized by $I\times J_n$ as follows. For $n=1$ one just chooses a homotopy constant near the boundary between $(k_1,\mathfrak{r}_1)$ and $(k_1',\mathfrak{r}'_1)$ so that at each $t$ the perturbation at the ends of $Z_1$ is the fixed one. This determines a family of data on the boundary of $I\times J_2$:
\begin{itemize}
\item on the edges $\partial I\times J_2$ we have the two families $(k_1,\mathfrak{r}_1)$ and $(k'_1,\mathfrak{r}'_1)$;
\item on the edges $I\times \partial J_2$ we have the families on $(Z_1\amalg Z_1)\amalg Z_2$ and $Z_2\amalg Z_1$ given by fixing the data on $Z_2$ and using the family $(l_1,\mathfrak{s}_1)$ on the $Z_1$ components.
\end{itemize}
We can extend this family on $\partial(I\times J_2)$ to the product as a family $(l_2,\mathfrak{s}_2)$ so that the moduli spaces are regular using Proposition \ref{fundtr}. The higher families $(l_n,\mathfrak{s}_n)$ are defined inductively is an analogous way. The maps $h_n$ are then defined by taking fibered products with the moduli spaces of $Z_n$ parametrized by $\mathfrak{q}_n$, after restricting to the suitable subcomplex on which it is defined. The map induced in homology by $f_1$ is just the map induced by the cobordism, so the result follows.
\end{proof}
\vspace{0.3cm}

\begin{remark}\label{ijp}
We have only discussed the \textit{from} version of the theory, but it is clear that the same construction leads to higher compositions (defined only on a subcomplex such that the inclusion is a quasi-isomorphism)
\begin{align*}
\check{m}_n&: \check{C}_{\bullet}^{\jmath}(Y)\otimes \Cj(S^3)\rightarrow \check{C}_{\bullet}^{\jmath}(Y)\\
\bar{m}_n&: \bar{C}_{\bullet}^{\jmath}(Y)\otimes \Cj(S^3)\rightarrow \bar{C}_{\bullet}^{\jmath}(Y)
\end{align*}
satisfying the $\Ainf$-relations. Furthermore the usual maps $i_*, j_*$ and $p_*$ are the first component of $\Ainf$-morphisms $\{i_n\}$, $\{j_n\}$ and $\{p_n\}$. The higher morphisms are obtained by considering the moduli spaces on the cobordism $W_n$ parametrized by our data, via the analogue of the formulas defining the usual maps (which are the case $W_0$, i.e. the product $I\times Y$). Here, to define the map $\check{m}_n$ it is essential that there are no irreducible solutions.
\end{remark}

\vspace{0.5cm}

Finally we discuss functoriality under a cobordism between based three manifolds $Y_0$, $Y_1$ with the respective embeddings of three balls $\iota^U_0$ and $\iota^U_1$. Suppose we are given a cobordism $X$ between them, which is also equipped with a proper embedding $\iota$ of $I\times B^3$ that restrictics to $\iota^U_i$ on the boundary of $I$. Using this additional data we can define (referring to a model independent of $Y$) the hypersurfaces on the cobordism with $n$ balls removed as in Figure \ref{cob}: $n$ of these are canonically identified with $Y_0$ (using the given embedding) and $n$ with $Y_1$.

\begin{figure}
  \centering
\def\svgwidth{\textwidth}
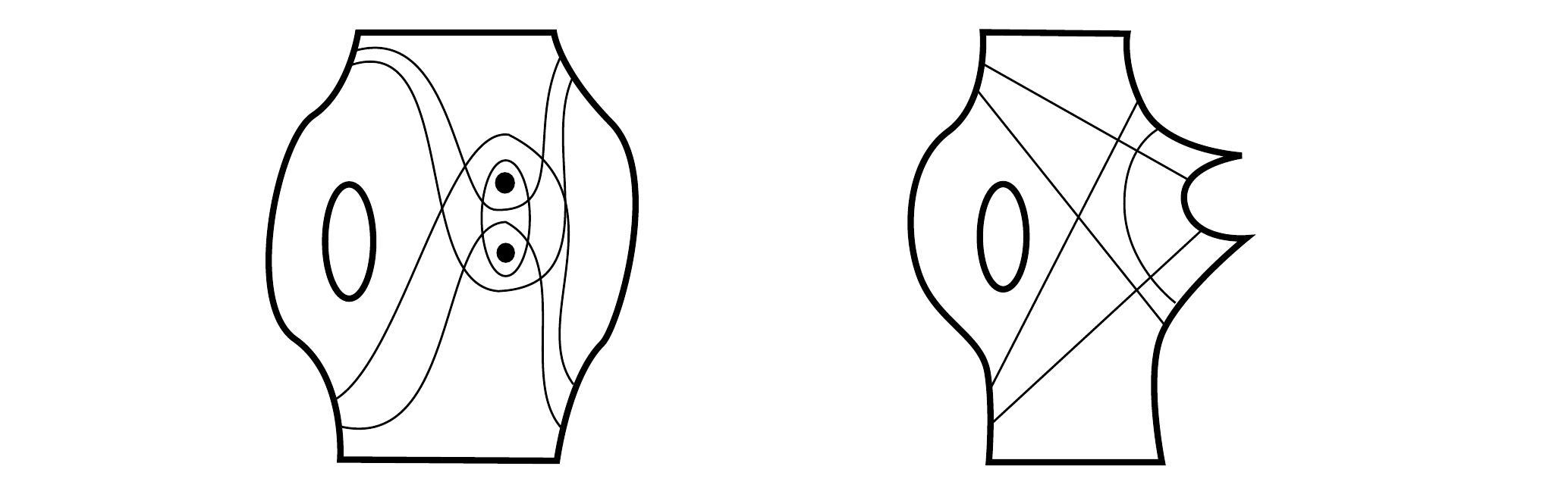
    \caption{The hypersurfaces in the cobordism $X$ with two balls removed.}
    \label{cob}
\end{figure} 

Again we can define a family of metrics and perturbations parametrized by the associahedron $K_n$ inductively, and by taking fibered products the maps (define on some suitable subcomplex whose inclusion induces a quasi-isomorphism)
\begin{equation*}
\hat{m}(X)_n:\Cj(Y_0, n)\rightarrow \Cj(Y_1)
\end{equation*} 
which satisfy the relations of an $\Ainf$-morphism of $\Ainf$-modules over $\Cj(S^3)$.

\vspace{1cm}
\section{The connected sum formula}\label{formula}

In this section we discuss the main result of the paper, Theorem \ref{conn}. We start by defining the suitable version of the $\Ainf$-tensor product that will be used. This is based on the construction we have discussed in Section \ref{algebra}, which has to be adapted to deal with transversality issues.
\par
Given two three manifolds $Y_0,Y_1$ equipped with self-conjugate spin$^c$ structures $\spin_0$ and $\spin_1$, fix data $\mathscr{D}_0$ and $\mathscr{D}_1$ so that the higher compositions are defined and the data on $Z_n$ agrees for every $n$. Define the subcomplex
\begin{equation*}
\Cj(Y_0,Y_1,n)\subset\Cj(Y_0)\otimes \hat{C}(S^3)^{\otimes n}\otimes \Cj(Y_1)
\end{equation*}
generated by elements $\sigma_0\otimes \tau_1\otimes \cdots\otimes \tau_n\otimes \sigma_1$ so that on which proper substring the corresponding higher composition is well defined. Of course the inclusion is a quasi-isomorphism. We define the chain complex $\hat{C}(Y_1)\boxtimes\left(\hat{C}(Y_2)\right)^{\mathrm{opp}}$ to be the vector space
\begin{equation*}
\bigoplus_{n=0}^{\infty} \Cj(Y_0,Y_1,n)
\end{equation*}
equipped with the differential
\begin{align*}
\partial(\sigma_0\otimes \tau_1\otimes\cdots\tau_n\otimes \sigma_2)&=\sum_{i=0}^{n} m_{i+1}(\sigma_0\otimes \tau_1\otimes\cdots\tau_{i-1})\otimes \tau_i\otimes \cdots \otimes\tau_n\otimes \sigma_1\\
&=\sum_{k=1}^{n}\sum_{l=1}^{n+1-l}\sigma_0\otimes \tau_1\otimes\cdots\otimes \mu_l(\tau_k\otimes\cdots\otimes \tau_{k+l-1})\otimes \otimes \tau_n\otimes \tau_1\\
&=\sum_{i=0}^{n}\sigma_0\otimes\tau_1\otimes\cdots\otimes m_{i+1}^{\mathrm{opp}}(\tau_{n-i+1}\otimes\cdots\otimes \tau_n\otimes \sigma_1).
\end{align*}
The algebraic results in Section \ref{algebra} hold with the same proofs for this modified version of the tensor product.
\\
\par
Given two objects in the based surgery precategory we now construct a chain map
\begin{equation*}
\Psi: \Cj(Y_1)\boxtimes\left(\Cj(Y_2)\right)^{\mathrm{opp}}\rightarrow \Cj(Y_1\#Y_2).
\end{equation*}
Again this will be defined on a subcomplex such that the inclusion induces a quasi-isomorphism. Consider the cobordism $X^{\#}$ described in Section \ref{HMconn} by attaching a $1$-handle to $I\times(Y_1\amalg Y_2)$. Fix regular metrics and perturbations on it coinciding with the given ones on the incoming ends. We also consider the cobordism $X^{\hash}_n$ obtained by removing $n$ standard balls from the attached $1$-handle.
\par
Following the scheme of the previous sections, we define a family of metrics and perturbations on $X^{\#}_n$ parametrized by the associahedron $K_{n+1}$ inductively as follows. For $n=0$ we have already defined the data on $W^{\#}$. Inductively, the cobordism $X^{\hash}_n$ contains a set of hypersurfaces dividing it into cobordisms on which the data has already been defined, see Figure \ref{family}. These can be constructed with a canonical identification with $Y_0$, $Y_1$ and $S^3$ using the same approach of Section \ref{HMconn}. In particular, we have defined the data on the boundary of $K_{n+1}$, and we can extend it so that near each face it corresponds to stretching along the corresponding hypersurface.
\begin{figure}
  \centering
\def\svgwidth{0.9\textwidth}
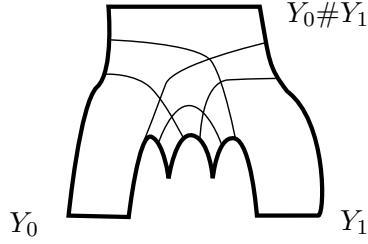
    \caption{The separating hypersurfaces in the cobordism $X^{\#}_2$.}
    \label{family}
\end{figure} 
\par
We define the map
\begin{equation*}
\Psi_n: \Cj(Y_0,Y_1,n)\rightarrow \Cj(Y_0\#Y_1).
\end{equation*}
by taking fibered products with the compactified moduli spaces on $(X^{\#}_n)^*$ parametrized by the family of metrics and pertubations. As usual we restrict to the subcomplex generated by the products of chains which are transverse to all the moduli spaces.
\par
It is straightforward to check at this point that $\Psi$ is actually a chain map, and our goal is to prove that this map is a quasi-isomorphism. To show this will take the approach outlined in Section \ref{HMconn}. In particular, for a fixed self-conjugate spin$^c$ structure $\spin_0$ on a three manifold $(Z,\jmath^U,\jmath^{\hash})$, we have two functors
\begin{align*}
\mathscr{F}(Y)&= H_*\left( \Cj(Y)\boxtimes \Cj(Z,\spin_0)^{\mathrm{opp}}\right)\\
\mathscr{G}(Y)&= \bigoplus_{[\spin\lvert_Z]=\spin_0}\HSf_{\bullet}(Y\# Z, [\spin]).
\end{align*}
Here on the $Y$ summand we sum over all spin$^c$ structures (even not self-conjugate ones), while we consider only the given on $Z$. It is straightforward to check that the construction above is compatible with these restrictions.
The map $\Psi$ defined above determines for each $Y$ a map
\begin{equation*}
\Psi(Y): \mathscr{F}(Y)\rightarrow \mathscr{G}(Y).
\end{equation*}
The main result of the paper, Theorem \ref{conn}, follows if we can show that the two conditions in Theorem \ref{floerf} hold. The first one is straightforward.
\begin{prop}
The map $\Psi(S^3)$ is an isomorphism.
\end{prop}
\begin{proof}
As $\HSf_{\bullet}(S^3)$ is $\Rin$, the Eilenberg-Moore spectral sequence associated to $\Cj(Y)\boxtimes \Cj(Z,\spin_0)^{\mathrm{opp}}$ from Lemma \ref{algEM} collapses at the second page, and we have that the final result is the zero filtration group $\Rin\otimes_{\Rin} \HSf_{\bullet}(Z,\spin_0)$. In particular each homology class is represented by an element of the form $1\otimes x$. As the map induced in homology by $\Psi_0$ is just the module multiplication, this element is mapped to $x$, and the result follows.
\end{proof}
\vspace{0.5cm}

The second condition requires more work, and we first briefly recall how the exact triangles in $\Pin$-monopole Floer homology are constructed (see \cite{Lin2} for more details). Let $Y_1,Y_2$ and $Y_3$ a surgery triple. Of the three surgery cobordisms $W_i$ from $Y_i$ to $Y_{i+1}$ exactly one is not spin, let us assume without loss of generality that it is $W_3$. Denote by $({C}_i,\partial_i)$ the from $\Pin$-monopole Floer complex of $Y_i$, and let $\{f_n^i\}_{n\in\mathbb{N}}$ be the family induced by $W_i$, so that in particular $f_1^i$ is the usual map induced by the cobordism. Then the composition $f_1^2\circ f_1^1$ is nullhomotopic by a nullhomotopy $h$ obtained by looking at a one parameter family of metrics on the composite $W_2\circ W_1$, see Figure \ref{h1}. The key point is that the composite is the blow up of a non-spin cobordism (indeed, $\bar{W}_3$), and in fact the analogous statement is false (in general) for the other two composites. This is because when stretching the blow-up to infinity there are no self-conjugate spin$^c$ structures involved so we can use Morse perturbations on $S^3$ and the contributions of the moduli spaces on $\overline{\mathbb{C}P}^2$ cancel in pairs.
\begin{figure}
  \centering
\def\svgwidth{0.9\textwidth}
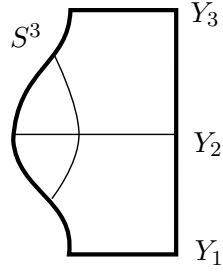
    \caption{The composite $W_2\circ W_1$ contains a separating hypersurface diffeomorphic to $S^3$. Indeed it is diffeomorphic to the blowup of $\bar{W}_3$.}
    \label{h1}
\end{figure} 

We can form the mapping cone $M_{f_1^1}$ whose underlying space is ${C}_2\oplus {C}_1$ with differential
\begin{equation*}
\partial=\begin{pmatrix}
\partial_2 & f_1^1\\
0 & \partial_1.
\end{pmatrix}
\end{equation*}
This fits in an exact triangle 
\begin{center}\label{cone1}
\begin{tikzpicture}
\matrix (m) [matrix of math nodes,row sep=2em,column sep=1.5em,minimum width=2em]
  {
  H_*({C}_2) && H_{*}(M_{f_1^1})\\
  &H_*(C_1) &\\};
  \path[-stealth]
  (m-1-1) edge node [above]{$i_*$} (m-1-3)
  (m-2-2) edge node [left]{$(f_1^1)_*$} (m-1-1)
  (m-1-3) edge node [right]{$p_*$} (m-2-2)  
  ;
\end{tikzpicture}
\end{center}
We can also form the iterated mapping cone $C$ whose underlying vector space is ${C}_3\oplus{C}_2\oplus {C}_1$ and has differential
\begin{equation*}
\partial=\begin{pmatrix}
\partial_3 & f_1^2 & h\\
0 & \partial_2 & f_1^1\\
0 & 0 & \partial_1.
\end{pmatrix}
\end{equation*}
This is a differential because $h$ is a nullhomotopy of $f_1^2\circ f_1^1$. The key result of \cite{Lin2} is that this chain complex is acyclic, so that the connecting homomorphism
\begin{equation*}
\delta_*: H(M_{f_1^1})\rightarrow H(C_3)
\end{equation*}
is an isomorphism, and we obtain the triangle
\begin{center}\label{cone2}
\begin{tikzpicture}
\matrix (m) [matrix of math nodes,row sep=2em,column sep=1.5em,minimum width=2em]
  {
  H_*({C}_2) && H_{*}(C_3)\\
  &H_*(C_1) &\\};
  \path[-stealth]
  (m-1-1) edge node [above]{$(f_1^2)_*$} (m-1-3)
  (m-2-2) edge node [left]{$(f_1^1)_*$} (m-1-1)
  (m-1-3) edge node [right]{$p_*\circ \delta_*^{-1}$} (m-2-2)  
  ;
\end{tikzpicture}
\end{center}
To show that the iterated mapping cone is acyclic, we define a map $G$ defined blockwise whose bottom left entry is $f_3^1$, the lower diagonal terms are the analogues of the homotopy $h$ for the other two composites, and the diagonal terms are obtained by counting solutions on the triple composites parametrized by a pentagon of metrics. We then show that the chain map $\partial\circ G+G\circ \partial$ (which is of course also nullhomotopic) is a quasi-isomorphism.
\\
\par
The homotopy $h$ described above can be enhanced to an actual $\Ainf$ nullhomotopy $\{h_i\}$ for the $\Ainf$-composition $f_1^2\circ f_1^1$. For example the higher homotopy $h_2$ is obtained by constructing a family of data parametrized by the associahedron $K_4$ by stretching along the hypersurfaces shown in Figure \ref{h2}. Again in the definition of these we use the extra data encoded in the morphisms of the based surgery precategory. Taking fibered products with the parametrized moduli spaces induces map on a subcomplex whose inclusion is a quasi-isomorphism. The main observation is that the stratum corresponding to $S^3$ being stretched to infinity does not contribute because the moduli spaces on $\overline{\mathbb{C}P}^2$ with a ball removed cancel in pairs (again because $W_3$ is not spin). The other four strata of $\partial K_4$ correspond to the sum
\begin{equation*}
h_1(m_2(\x,r))+m_2(h_1(\x),r)+f_2^2(f_1^1(\x),r)+f_1^2(f_2^1(\x,r)),
\end{equation*}
so $h_2$ provides a nullhomotopy of this map.
\begin{figure}
  \centering
\def\svgwidth{0.9\textwidth}
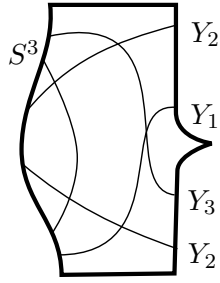
    \caption{The five hypersurfaces defining the family of metrics and perturbations of the composition $W_2\circ W_1$ with a ball removed.}
    \label{h2}
\end{figure} 
\\
\par
The next is the last thing to check in order to prove Theorem \ref{conn}.
\begin{prop}
The functors $\mathscr{F}$ and $\mathscr{G}$ are Floer functors and $\Psi$ is a natural transformation between them.
\end{prop}
\begin{proof}
It is clear that $\mathscr{G}$ is a Floer functor, as it follows from the standard surgery exact triangle by restricting to a fixed self-conjugate spin$^c$ structure on $Z$. To show that $\mathscr{F}$ is also a Floer functor we adopt the same approach as the proof of exactness of the surgery exact triangle. We will denote for $\Cj(Z,\spin_0)^{\mathrm{opp}}$ by $N$ and treat our objects as actual $\Ainf$-algebras and modules for simplicity of notation. The mapping cone of the chain map $f^1_1$ has a natural $\Ainf$-module structure given by
\begin{equation*}
m_n\left( (\x_1,\x_2), r_1,\cdots, r_{n-1}  \right)=\left(m_n(\x_1,r_1,\dots, r_{n-1}), f_n^1(\x_1,r_1,\dots, r_{n-1})+m_n(\x_2,r_1,\dots, r_{n-1})\right)
\end{equation*} 
There is a natural identification between the chain complexes $M_{(f^1\boxtimes 1_N)}$ and $M_{f^1}\boxtimes N$. Similarly the composition of the maps $f^2\boxtimes 1_N$ and $f^1\boxtimes 1_N$ is nullhomotopic, the nullhomotopy given by $h\boxtimes 1_N$.  Hence we can form the iterated mapping cone, which can also be identified with $C\boxtimes N$, where we consider the natural $\Ainf$-structure on $C$. Now the same construction outlined above shows that this tensor product is acyclic, so that the connecting homomorphism
\begin{equation*}
H_*(M_{f^1\boxtimes 1_N})\rightarrow H_*(C_3\boxtimes N)
\end{equation*}
is an isomorphism. Hence we get an exact triangle between the Floer groups
\begin{center}\label{coneN}
\begin{tikzpicture}
\matrix (m) [matrix of math nodes,row sep=2em,column sep=1.5em,minimum width=2em]
  {
  H_*({C}_2\boxtimes N) && H_{*}(C_3\boxtimes N)\\
  &H_*(C_1\boxtimes N) &\\};
  \path[-stealth]
  (m-1-1) edge node [above]{$(f^2\boxtimes 1_N)_*$} (m-1-3)
  (m-2-2) edge node [left]{$(f^1\boxtimes 1_N)_*$} (m-1-1)
  (m-1-3) edge node [right]{$p_*\circ \delta_*^{-1}$} (m-2-2)  
  ;
\end{tikzpicture}
\end{center}
It is customary to show that the construction is well defined (i.e. independent of the additional choices in a canonical way) so that $\mathscr{F}$ is a Floer functor.
\par
We need now to show that the map $\Psi$ is a natural transformation. Consider the diagram
\begin{center}\label{cone2}
\begin{tikzpicture}
\matrix (m) [matrix of math nodes,row sep=2.5em,column sep=2.5em,minimum width=2em]
  {
  \hat{C}_{\bullet}(Y_1\#Z) & \hat{C}_{\bullet}(Y_2\#Z)\\
  C_1\boxtimes N & C_2\boxtimes N\\};
  \path[-stealth]
  (m-1-1) edge node [above]{$F^1$} (m-1-2)
  (m-2-1) edge node [left]{$\Psi(Y_1)$} (m-1-1)
  (m-2-1) edge node [above]{$f^1\boxtimes 1_N$} (m-2-2)
  (m-2-2) edge node [right]{$\Psi(Y_2)$} (m-1-2)  
  ;
\end{tikzpicture}
\end{center}
where the top map is the one induced by the cobordism $X^{\hash}_1\# (I\times Z)$ (recall the the connected sum is performed away from the knot on which we do surgery). Here we use the embeddings $\iota^{\hash}$ and $\jmath^{\hash}$ to perform the connected sum. This diagram commutes up to a homotopy $K$ that we now describe, so that the maps induced in homology commute.
 The map
\begin{equation*}
k_n: C_1\otimes \Cj(S^3)\otimes N
\end{equation*}
is defined by looking at moduli spaces parametrized by a family of metrics and perturbations on the cobordism $(W_1\hash I\times Z)\circ X_{\#}$ with $n-1$ open balls removed, see Figure \ref{nat1} for when $n$ is $1$.
\begin{figure}
  \centering
\def\svgwidth{0.9\textwidth}
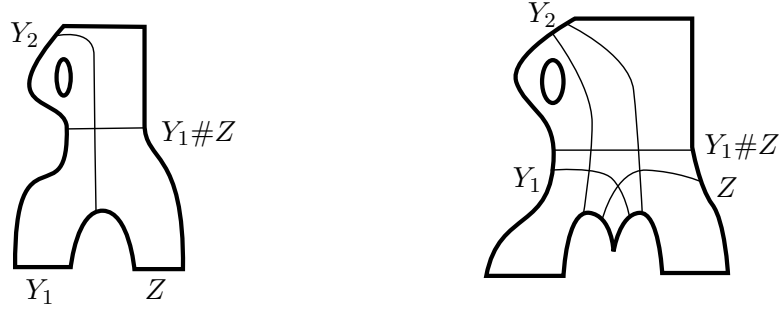
    \caption{The hypersurfaces on the composite $(W_1\hash I\times Z)\circ X^{\#}_n$. The cobordism $W_1\hash I\times Z$ is the upper half in the pictures.}
    \label{nat1}
\end{figure} 
The first case on the right defines a map $k_1$ satisfying
\begin{equation*}
\partial k_1(\x,\mathbf{z})+k_1(\partial \x,\mathbf{z})+k_1(\x,\partial \mathbf{z})=F^1(\Psi(Y_1)(\x,\mathbf{z}))+\Psi(Y_2)(f_1^1(\x),\mathbf{z}),
\end{equation*}
i.e. it is a homotopy among the composition of the chain maps. Similarly the figure on the right gives rise to the relation
\begin{align*}
\partial k_2(\x,r,\mathbf{z})&= k_2(\partial \x,r,\mathbf{z})+k_2(\x,\partial r,\mathbf{z})+k_2(\x,r,\partial \mathbf{z})\\
&+ k_1(m_2(\x,r),\mathbf{z})+k_1(\x,m_2(r,\mathbf{z}))\\
&+ \Psi(Y_2)(f_1^1(\x),r,\mathbf{y})+ \Psi(Y_2)(f_2^1(\x,r),\mathbf{z})+ F^1(\Psi_1(Y_1)(\x,r,\mathbf{z})).
\end{align*}
Notice that the first two lines involve the differential of $\x\otimes r\otimes \mathbf{z}$ in the chain complex so that we have a chain homotopy between the two composites. \par
Using this map we can we can construct the chain map
\begin{align*}
\mathrm{Cone}(\Psi)&:\mathrm{Cone}(f_1\boxtimes 1_N)\rightarrow \mathrm{Cone}(F^1)\\
\mathrm{Cone}(\Psi)&=
\begin{pmatrix}
\Psi(Y_2) & K \\
0  & \Psi(Y_1)
\end{pmatrix}.
\end{align*}
The map induced in homology is independent of the choices made, and it commutes with the others in the mapping cone exact triangles. This again follows because our construction are performed using the extra data provided by the morphisms in the based surgery precategory and the space of metrics and perturbations is contractible.
\begin{center}\label{cone2}
\begin{tikzpicture}
\matrix (m) [matrix of math nodes,row sep=2.5em,column sep=2.5em,minimum width=2em]
  {
   \cdots&\mathscr{G}(Y_2)& H_*(\mathrm{Cone}(F^1))& \mathscr{G}(Y_1) & \cdots\\
  \cdots & \mathscr{F}(Y_2) & H_*(\mathrm{Cone}(f_1\boxtimes 1_N)) & \mathscr{F}(Y_1) &\cdots \\ };
  \path[-stealth]
  (m-1-1) edge node [above]{$(F^1)_*$} (m-1-2)
  (m-1-2) edge node [above]{} (m-1-3)
	  (m-1-3) edge node [above]{} (m-1-4)  
  (m-1-4) edge node [above]{$(F^1)_*$} (m-1-5)
  (m-2-4) edge node [left]{$\Psi(Y_1)$} (m-1-4)
  (m-2-3) edge node [left]{$\mathrm{Cone}(\Psi)$} (m-1-3)
  (m-2-2) edge node [left]{$\Psi(Y_2)$} (m-1-2)  
  (m-2-1) edge node [above]{$(f^1)_*$} (m-2-2)
  (m-2-2) edge node [above]{} (m-2-3)
	  (m-2-3) edge node [above]{} (m-2-4)  
  (m-2-4) edge node [above]{$(f^1)_*$} (m-2-5)
 
  ;
\end{tikzpicture}
\end{center}
The same identical construction applies to the iterated mapping cones of the maps. In particular this shows that the connecting homomorphisms form commutative diagram
\begin{center}
\begin{tikzpicture}
\matrix (m) [matrix of math nodes,row sep=2.5em,column sep=2.5em,minimum width=2em]
  {
  H_*(\mathrm{Cone}(F^1))& \mathscr{G}(Y_3)\\
  H_*(\mathrm{Cone}(f^1\boxtimes 1_N)) & \mathscr{F}(Y_3)\\};
  \path[-stealth]
  (m-1-1) edge node [above]{$\Delta_*$} (m-1-2)
  (m-2-1) edge node [left]{$\mathrm{Cone}(\Psi)$} (m-1-1)
  (m-2-1) edge node [above]{$\delta_*$} (m-2-2)
  (m-2-2) edge node [right]{$\Psi(Y_3)$} (m-1-2)  
  ;
\end{tikzpicture}
\end{center}
so that $\Psi$ is a natural transformation.
\end{proof}

\vspace{1cm}
\section{The Eilenberg-Moore spectral sequence and Massey products}\label{massey}
In this section we discuss some basic properties of the Eilenberg-Moore spectral sequence described in Corollary \ref{EM}. This arises from Theorem \ref{conn} by taking the natural filtration in the $\Ainf$-tensor product, see Lemma \ref{algEM}.
Our goal is to describe the higher differentials and the module structure terms of Massey (bi)products. This fact should not be surprising and in fact it has already interesting application in classic topics in algebraic topology (see \cite{McC}). Before doing so, we quickly discuss the proof of Corollary \ref{image}.
\begin{proof}[Proof of Corollary \ref{image}.]
This readily follows by identifying the map
\begin{equation*}
\hat{m}(X^{\#}):\Cj(Y_1)\otimes \Cj(Y_2)\rightarrow\Cj(Y_1\# Y_2)
\end{equation*}
under the isomorphism of Theorem \ref{conn} as the inclusion in the $\Ainf$-tensor product $\hat{C}(Y_1)\boxtimes \left(\hat{C}(Y_2)\right)^{\mathrm{opp}}$, and looking at the induced map on the associated spectral sequence. From this the inequalities for the correction terms in the case of homologies spheres readily follows because the maps in the \textit{bar} flavor
\begin{equation*}
\HSb(W^{\#}):\HSb_{\bullet}(Y_1)\otimes \HSb_{\bullet}(Y_2)\rightarrow\HSb_{\bullet}(Y_1\# Y_2)
\end{equation*}
is identified (after a grading shift) with the multiplication in the ring $\Rin$.
\end{proof}

\vspace{0.8cm}

For simplicity, we will focus on the differentials $d_2$ and $d_3$ which are the ones that will arise in our explicit computations. The general case is analogous. From this perspective, it is good to consider the $E^2$-page of the Eilenberg-Moore spectral sequence as the homology of the bar resolution of the two modules, see Section \ref{algebra}.
\begin{prop}\label{masseyprod}
Suppose we are given an element in the $E^2$ page which has a representative of the form $\x\otimes r_1\otimes\cdots\otimes r_n\otimes \y$ so that the triple Massey products of any three consecutive terms is defined (i.e. all products of consecutive terms vanish). Then we have
\begin{align*}
d_2([\x\otimes r_1\otimes\cdots\otimes r_n\otimes \y])&=[ \langle\x,r_1,r_2\rangle\otimes r_3\otimes \cdots r_n\otimes \y=\\
&+\sum_{i=1}^{n-2}\x\otimes r_1\otimes \cdots\otimes \langle r_i,r_{i+1},r_{i+2}\rangle\otimes \cdots\otimes \y=\\
&+\x\otimes r_1\otimes\cdots\otimes \langle r_{n-1}, r_{n},\y\rangle].
\end{align*}
Similarly, if the fourfold Massey products of any four consecutive elements is defined, then the element survives to the $E^3$ page and we have
\begin{align*}
d_3([\x\otimes r_1\otimes\cdots\otimes r_n\otimes \y])&= \langle \x,r_1,r_2,r_3\rangle\otimes r_4\otimes \cdots r_n\otimes \y+\\
&+\sum_{i=1}^{n-3}\x\otimes r_1\otimes \cdots\otimes \langle r_i,r_{i+1},r_{i+2},r_{i+3}\rangle \otimes \cdots\otimes y\\
&+\x\otimes r_1\otimes\cdots\otimes \langle r_{n-2}, r_{n-1},r_{n},\y\rangle.
\end{align*}
\end{prop}
Notice that the image of the differential is not necessarily a well defined homology class itself. Indeed, the statement of the result includes implicitly the fact that this element is well defined in the respectively $E^2$ or $E^3$ page of the spectral sequence.
\begin{proof}
We discuss the proof in the first statement in the case $n=2$, as the other ones are more complicated only from a notational point of view. This is analogous to the standard staircase argument for the spectral sequence of a double complex, see for example Lemma 8.30 in \cite{McC}. Consider a chain representative $\x\otimes r_1\otimes r_2\otimes \y$, so that in particular $\x$, $r_1$, $r_2$ and $\y$ are cycles. As consecutive products are zero in homology, there exist classes $s_1,s_2,s_3$ so that
\begin{equation*}
\partial_0 s_1=m_2(\x,r_1),\quad \partial_0 s_2=\mu_2(r_1,r_2),\quad \partial_0 s_3=m_2(r_2,\y).
\end{equation*}
Consider the chain
\begin{equation*}
c=\x\otimes r_1\otimes r_2\otimes \y+ s_1\otimes r_2\otimes \y+ \x\otimes s_2\otimes \y+\x\otimes r_1\otimes s_3,
\end{equation*}
This defines the same class in the $E^2$-page and has total boundary
\begin{align*}
\partial c&= \left(\mu_2(s_1,r_2)+m_2(\x,s_2)+m_3(\x,r_1,r_2)\right)\otimes \y\\
&+\x\otimes\left(m_2(s_2,\y)+\mu_2(r_1,s_3)+m_3(r_1,r_2,\y)\right)+\partial_0 (s_1\otimes s_3)
\end{align*}
The result then follows by tracking the definition of the differential in the $E^2$ page of the spectral sequence.
\end{proof} 

\vspace{0.5cm}

We now discuss the module structure. To be concise, we will treat our objects as if they were genuine $\Ainf$-bimodules, so forgetting about transversality issues. These can be dealt with as in the previous sections. We have the following result.
\begin{prop}
For an admissible choice of data, for each $i,j\geq 1$ there are compositions
\begin{equation*}
m_{i,j}:\Cj(S^3)^{\otimes i-1}\otimes \Cj(Y)\otimes \Cj(S^3)^{\otimes j-1}
\end{equation*}
making $\Cj(Y)$ into an $\Ainf$-bimodule. This structure is well defined up to $\Ainf$-equivalence. Furthermore, the isomorphism $\Psi$ of Theorem \ref{conn} is an isomorphism of $\Ainf$-bimodules.
\end{prop}
\begin{proof}
Consider the cobordism $W_{i,j}$ obtained by removing $i+j-1$ open balls, $i-1$ of which we consider on the left and $j-1$ on the right. The compositions $m_{ij}$ are defined by taking fibered products with the moduli spaces parametrized by associahedra whose faces correspond to separating hypersurfaces, see Figure \ref{bimodule}.
\end{proof}
\begin{figure}
  \centering
\def\svgwidth{0.9\textwidth}
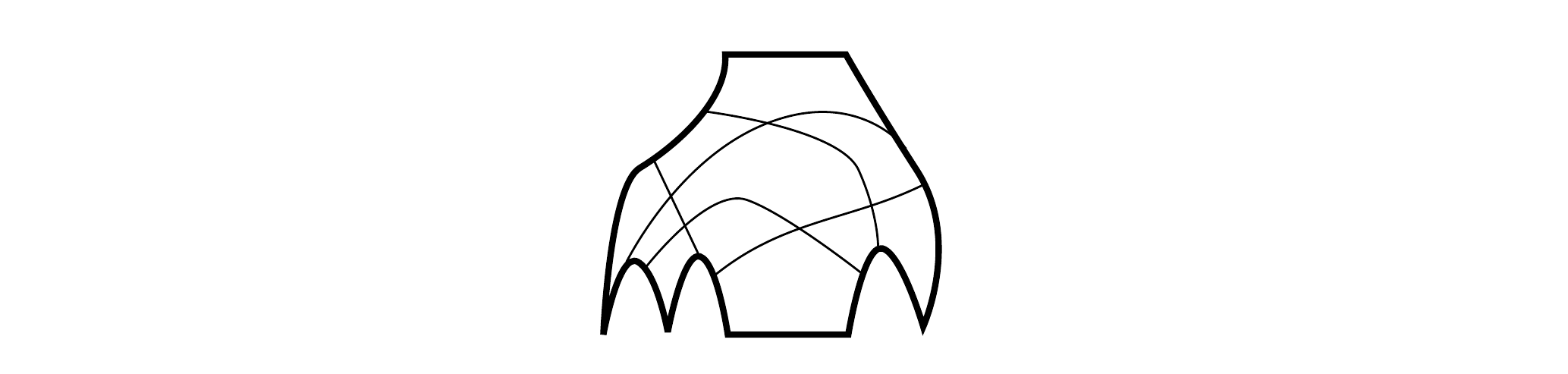
    \caption{The separating hypersurfaces in the cobordism $W_{2,1}$. These correspond to the faces of the associahedron parametrizing the data used to define the composition $m_{2,1}$.}
    \label{bimodule}
\end{figure} 

As a particular case we can consider the structure of $\Rin$-module induced in homology. A class $[r]\in\Rin$ induces a filtration preserving chain map
\begin{equation*}
\Psi(r): \Cj(Y_1)\boxtimes \Cj(Y_2)\rightarrow \Cj(Y_1)\boxtimes \Cj(Y_2)
\end{equation*}
given by
\begin{equation*}
\x\otimes r_1\otimes \cdots \otimes r_n\otimes \y\mapsto\sum_{i=1}^n m_{1,n+1} (r,\x,r_1,\cdots , r_i)\otimes r_{i+1}\otimes \cdots\otimes \y,
\end{equation*}
so we can try to understand the module structure by means of the Eilenberg-Moore spectral sequence. An immediate result is that the filtration on $\HSf_{\bullet}(Y_1\# Y_2)$ is a filtration of $\Rin$-submodules. On the other hand, by taking the associated graded module much information about the $\Rin$-module structure is lost. A way to recover it is by means of the following result, which is proved in the same way as Proposition \ref{masseyprod}. We state the simplest cases that will be useful in our computations, but the statement is not hard to generalize.
\begin{lemma}\label{moduleEM}
Suppose that an element in the $E^{\infty}$ page is represented by $\x\otimes r_1\otimes\y$. If all the products of consecutive elements in the list $r,\x,r_1,\y$ are zero, then the action of $r$ on it is given by $\langle r, \x, r_1\rangle\otimes \y$. 
\end{lemma}

\vspace{1cm}
\section{An explicit computation}\label{compute}
We consider the Brieskorn homology sphere $Y=\Sigma(2,3,11)$. From \cite{Lin2} that we have that the isomorphism of graded $\Rin$-modules
\begin{equation*}
\HSt_{\bullet}(Y)=\V^+_2\oplus \V^+_1\oplus\V^+_4=\mathcal{S}^+_{2,0,0}
\end{equation*}
where the action of $Q$ is a surjection from the first tower onto the second and from the second tower onto the third. In particular the Manolescu correction terms are $2,0,0$.
\begin{prop}\label{2Y}
We have the isomorphism of graded $\Rin$-modules
\begin{equation*}
\HSt_{\bullet}(2Y)=(\V^+_2\oplus \V^+_5\oplus\V^+_4)\oplus \ztwo^2\langle3\rangle
\end{equation*}
where the action of $Q$ is a surjection from the first tower to the second and from the second tower to the third. Furthermore the $Q$ action from the last summand to the first tower is not trivial. In particular the Manolescu's correction terms are $2,2,0$.
\end{prop}
\begin{proof}
We have that $\HSf_{\bullet}(Y)=N\langle1\rangle$ where $N$ is the $\Rin$-module we discussed in Example \ref{S2311}. Similarly, we know that
\begin{equation*}
\HMf_{\bullet}(Y)=\ztwo[[U]]\langle1\rangle\oplus \ztwo\langle1\rangle.
\end{equation*}
The usual connect sum formula in Theorem \ref{connectusual} implies then that
\begin{equation*}
\HMf_{\bullet}(2Y)=\ztwo[[U]]\langle3\rangle\oplus \ztwo^3\langle3\rangle\oplus\ztwo\langle2\rangle.
\end{equation*}
It follows by looking at the Gysin exact sequence (see the last Section of \cite{Lin2} for the details) that there are only two possibilities for $\HSf_{\bullet}(2Y)$, namely
\begin{equation*}
(\ztwo[[V]]\langle-1\rangle\oplus \ztwo[[V]]\langle2\rangle\oplus \ztwo[[V]]\langle1\rangle)\oplus \ztwo^2\langle3\rangle,
\end{equation*}
where the action of $Q$ sends the first tower to the second and from the second to the third, or
\begin{equation*}
\Rin\langle 3\rangle\oplus \ztwo^2\langle 3\rangle\oplus \ztwo\langle 2\rangle.
\end{equation*}
We claim that $\HSf_3(2Y)$ has rank two so that the first case holds. The main result follows then by Poincar\'e duality (the additional $Q$ action follows by applying the Gysin exact sequence). To prove the claim, we look at the $E^2$-page of the Eilenberg-Moore spectral sequence. We already described a simple two periodic projective resolution of $N\langle1\rangle$. By tensoring it over $\Rin$ with $N\langle1\rangle$ itself and taking homology we can identify $\Tor_{\Rin}(N\langle1\rangle, N\langle1\rangle)$ to be as follows in bidegrees near $(0,3)$.
\begin{center}
\begin{tikzpicture}
\matrix (m) [matrix of math nodes,row sep=0.1em,column sep=0.5em,minimum width=0.1em]
  { \ztwo& \cdot & \cdot &\cdot& \cdot& \cdot\\
  \cdot & \ztwo &\cdot & \cdot & \cdot & \cdot\\
  {\ztwo^2} & \cdot & \cdot & \cdot & \cdot & \cdot\\
  \cdot &\cdot& \ztwo &\cdot &\cdot&\cdot\\
  \ztwo & \ztwo & \cdot & \ztwo & \cdot & \cdot\\
  \ztwo &\cdot& \ztwo &\cdot &\cdot&\cdot\\
  \ztwo &\cdot&\cdot&\cdot&\ztwo&\cdot\\
  \cdot& \cdot&\cdot&\ztwo &\cdot& \ztwo \\
    \ztwo &\cdot&\cdot&\cdot&\ztwo&\cdot\\
    \ztwo & \cdot & \cdot & \cdot & \cdot & \cdot\\
    \ztwo & \cdot & \cdot & \cdot & \cdot & \ztwo\\
};
\end{tikzpicture}
\end{center}
Here the top left element has bidegree $(0,3)$, and up to grading shift all the remaining group, except the first column, is isomorphic to the corresponding page of $\Tor_{\Rin}(\ztwo,\ztwo)$ in Example \ref{FF}. On the other hand, we know from general facts that the first column can be identified as $N\otimes_{\Rin}N\langle3\rangle$, so it can be identified as a $\Rin$-module with
\begin{equation*}
\left(\ztwo[[V]]\langle-1\rangle\oplus \ztwo[[V]]\langle-2\rangle\oplus \ztwo[[V]]\langle1\rangle\right)\oplus \ztwo\langle3\rangle\oplus \ztwo\langle 1\rangle,
\end{equation*}
where the $Q$ action is an isomorphism from the first tower to the second and is injective from the second to the third. Notice that the disposition of the non trivial groups implies that the $\ztwo$ summands in bidegree $(0,3)$ and $(2,1)$ are not involved in differentials, so that they survive till the $E^{\infty}$ page. They are the only non trivial groups in their antidiagonal, so that the rank of $\HSf_3(2Y)$ is $2$ and the result follows.
\end{proof}
\vspace{0.5cm}
Unlike the case of $2Y$, in the case of $3Y$ the Manolescu's invariants can be computed by means of the Gysin exact sequence (even though the group itself cannot be), and in particular they are $4,2,2$. This follows from the fact that
\begin{equation*}
\HMf_{\bullet}(3Y)=\ztwo[[U]]\langle5\rangle\oplus \ztwo^7\langle5\rangle\oplus\ztwo^5\langle4\rangle\oplus\ztwo\langle3\rangle,
\end{equation*}
and the fact that in this case the parity is right to imply that there is an element in the top grading in the tower. These computations are sufficient to prove Theorem \ref{nonadditive}. 

\begin{proof}[Proof of Theorem \ref{nonadditive}]
Suppose $\varepsilon=a\alpha+b\beta+c\gamma$ is a linear combination that defines a homomorphism. By $\varepsilon(-Y)=-\varepsilon(Y)$ and evaluating is some basic examples (see \cite{Lin2}) we obtain that $a=c$. Applying then $\varepsilon(2Y)=2\varepsilon(Y)$ we have then that also $a=b$. The fact that $a(\alpha+\beta+\gamma)$ is not an homomorphism then follows from the fact that $\varepsilon(3Y)\neq 3\varepsilon(Y)$.
\end{proof}

\vspace{0.5cm}

As we knot the group the spectral sequence converges to, we can tell that there has to be some differentials. Recall that $d_n$ has bidegree $(-n-1,n)$. In particular:
\begin{itemize}
\item of the two $\ztwo$ summands at $(2,0)$ and $(3,-1)$ exactly one cancels with a generator of the $\ztwo^2$ at $(0,1)$. This follows from the fact that both $\HSf_2(2Y)$ and $\HSf_1(2Y)$ have rank one.
\item the other groups in the first column are not involved by the differentials because their image in the $E^{\infty}$ corresponds to the image of the map
\begin{equation*}
\hat{m}(X^{\hash}): \HSf_{\bullet}(Y)\otimes \HSf_{\bullet}(Y)\rightarrow \HSf_{\bullet}(2Y),
\end{equation*}
as in Corollary \ref{image}. In particular we have that in the (infinitely many) configurations of four groups
\begin{center}\label{four}
\begin{tikzpicture}
\matrix (m) [matrix of math nodes,row sep=0.1em,column sep=0.5em,minimum width=0.1em]
  {   \ztwo & \cdot & \cdot & \cdot & \cdot\\
   \cdot& \ztwo &\cdot &\cdot&\cdot\\
   \cdot&\cdot&\cdot&\ztwo&\cdot\\
   \cdot&\cdot&\cdot &\cdot& \ztwo \\
};
\end{tikzpicture}
\end{center}
the two upmost cancel with the two bottommost.
\end{itemize}
To determine the actual differentials, we need the following result.
\begin{lemma}\label{m3vanish}
Let $Y$ be a homology sphere, and fix an identification of $\Rin$-modules of $\HSb_{\bullet}(Y)$ with $\ztwo[V^{-1},V]][Q]/(Q^3)$. Then the for each $i\in\mathbb{Z}$ and $j,k\in \mathbb{N}$ the triple Massey products
\begin{equation*}
\langle Q^2V^i, Q^2 V^j, Q V^k\rangle,\quad\langle Q^2V^i, Q V^j, Q^2 V^k\rangle,\quad \langle QV^i, Q^2 V^j, Q^2 V^k\rangle
\end{equation*}
are well defined in $\HSb_{\bullet}(Y)$ and zero.
\end{lemma}
Notice that for degree reasons the triple Massey product is either zero or the element $V^{i+j+k+1}$.
\begin{proof}
The heuristic behind the result is that for the first product we can factor the first element into $Q V^i\cdot Q$ and use the $\Ainf$-relations to write it as a sum of terms which vanish. For simplicity we will assume $i,j,k$ to be zero, as the proof will be the same in general. We only have to deal with reducible solutions. Fix cycles $r_1,r_2,r_3,r_4$ representing the classes $Q$, $Q$, $Q^2$ and $Q$. Notice that the cycles $\mu_2(r_2,r_3)$ and $\mu_2(r_3,r_4)$ are automatically zero at the chain level because in our definition of the Morse-Bott chain complex we are quotienting out by small chains, see the quick review in \cite{Lin2}. Hence, fixing a chain $s_1$ so that
\begin{equation*}
\partial s_1= (m_2(r_1,r_2),r_3),
\end{equation*}
the triple Massey product is the class of
\begin{equation*}
m_3(m_2(r_1,r_2),r_3,r_4)+ m_2(s_1,r_4).
\end{equation*}
Using the $\Ainf$ relations \ref{infalg} and the fact that the products involved vanish at the chain level we see that this cycle is cobordant to
\begin{equation*}
m_2\left((m_3(r_1,r_2,r_3)+s_1), r_4\right)+ m_2\left(r_1, \mu_3(r_2,r_3,r_4)\right).
\end{equation*}
The expression $m_3(r_1,r_2,r_3)+s_1$ is a cycle and defines the zero class in homology for dimensional reasons. Similarly $\mu_3(r_2,r_3,r_4)$ defines the triple Massey product $\langle Q,Q^2,Q\rangle$ so it vanishes because the group is trivial in that dimension.
\par
The rest of the result follows in a similar way. One shows for example that
\begin{equation*}
\langle Q^2V^i, Q V^j, Q^2 V^k\rangle=\langle Q^2V^i, Q^2 V^j, Q V^k\rangle,
\end{equation*}
hence is zero.
\end{proof}

From this we can easily determine the differentials.
\begin{prop}
In the spectral sequence of Proposition \ref{2Y}, the differential $d_2$ on the $E^2$ page vanishes, and the spectral sequence collapses at the $E^4$ page.
\end{prop}
\begin{proof}
We can fix the identification in Lemma \ref{m3vanish} so that the element in $\HSf_{\bullet}(Y)$ of top degree is mapped to $Q^2$, and denote each element by its image under $p_*$. It is the straighforward to check that the element $Q^2\otimes Q^2\otimes Q\otimes Q^2$ is a generator of the $\ztwo$ summand of bidegree $(2,0)$. The triple Massey products of consecutive terms vanish because of Lemma \ref{m3vanish} and functoriality (the map $p_*$ is a morphism of $\Ainf$-modules, see Remark \ref{ijp}), so by Lemma \ref{masseyprod} the differential $d_2$ is zero. More in general, the $\ztwo$ summand in position $(2+2n, 3n)$ is generated by
\begin{equation*}
Q^2\otimes (Q^2\otimes Q)^{\otimes n}\otimes Q^2
\end{equation*}
so again its differential $d_2$ is zero by Lemma \ref{masseyprod} and the fact that $\langle Q,Q^2,Q\rangle$ vanishes for dimensional reasons. This implies that the differential $d_3$ is non zero on both the bottom groups in the configuration described above.
\end{proof}

\begin{proof}[Proof of Corollary \ref{nonformal}]
A non vanishing Massey product is $\langle Q^2,Q,Q^2,Q\rangle$(which is $V$). This follows from the fact that the $\ztwo$ summand of bidegree $(3,-1)$ is generated by $Q\otimes Q\otimes Q^2\otimes Q\otimes Q$. It follows from Lemma \ref{masseyprod} that $d_3$ of this class is
\begin{equation*}
\langle Q^2,Q,Q^2,Q\rangle\otimes Q^2+Q^2\otimes \langle Q,Q^2,Q,Q^2\rangle.
\end{equation*}
As this has to be non zero, the fourfold Massey product $\langle Q^2,Q,Q^2,Q\rangle$ (which is well defined as an actual homology class) is non zero, hence the product $\langle Q^2,Q,Q^2,Q\rangle $ is $V$ in $\HSb_{\bullet}(Y)$ by functoriality. On the other hand the computation of this only relies on reducible solutions, so the same result holds for $S^3$.
\end{proof}

\vspace{0.5cm}
We conclude by discussing the $\Rin$-module structure. We know that the element $Q^2\otimes Q^2\otimes Q\otimes Q^2$ is a generator of the bottom element of the middle tower, so the action of $Q$ and $V$ on it are the respectively the generators of the $\ztwo$ summands in bidegree $(0,1)$ and $(0,-2)$. This are respectively given by $[V\otimes Q^2]=[Q^2\otimes V]$ and $[V\otimes QV]=[QV\otimes V]$.
\par
The action of $Q$ readily follows from Lemma \ref{moduleEM} and the non vanishing of the Massey $(2,3)$-byproduct $\langle Q,Q^2,Q,Q^2\rangle$. The last fact follows by functoriality from the non-vanishing of $\mu_4(Q,Q^2,Q,Q^2)$, as when dealing uniquely with reducibles we can identify the moduli spaces involved in the construction.
\par
For action of $V$ we cannot rely on Lemma \ref{moduleEM} or its generalizations as the class $V$ has non-zero product with $Q^2$. On the other hand we can understand the action in a similar way. On the $E^1$-page the action of $V$ sends $Q^2\otimes Q^2\otimes Q\otimes Q^2$ to $Q^2V\otimes Q^2\otimes Q\otimes Q^2$. This element is $d_1$ of $QV\otimes Q\otimes Q^2\otimes Q\otimes Q^2$. As the triple Massey products of each consecutive triple genuinely vanishes, the same staircase argument of Lemma \ref{masseyprod} lets us identify the action of $V$ on the class we started with with $d_4$ of this class, which is given by $QV\otimes \langle Q,Q^2,Q,Q^2\rangle$, hence the result.

\vspace{1cm}
\bibliographystyle{alpha}
\bibliography{biblio}

\end{document}